\documentclass[a4paper,11pt]{amsart}
\usepackage{amsmath}
\usepackage{url,graphicx}
\usepackage{amssymb, color, pstricks-add, manfnt}
\usepackage{amsmath, amsthm,  amsfonts}
\usepackage{mathrsfs}
\usepackage{cite}
\usepackage{enumerate}

\theoremstyle{plain}
\newtheorem{Theorem}{Theorem}[section]
\newtheorem{Lemma}{Lemma}[section]

\newtheorem{Corollary}{Corollary}[section]

\newtheorem{Example}{Example}[section]

\theoremstyle{remark}
\newtheorem{remark}{Remark}

\numberwithin{equation}{section}
\numberwithin{figure}{section}
\numberwithin{remark}{section}

\usepackage{hyperref}
\usepackage{cite}

\begin{document}
	
	\title{Entire subsolutions of a kind of $k$-Hessian type equations with gradient terms}
	
	\author{JingWen Ji}
	\address{College of Mathematics and Statistics, Nanjing University of Information Science and Technology, Nanjing 210044, P.R. China}
	\email{20201215013@nuist.edu.cn}
	
	\author{Feida Jiang*}
	\address{School of Mathematics and Shing-Tung Yau Center of Southeast University, Southeast University, Nanjing 211189, P.R. China}
	\email{jiangfeida@seu.edu.cn}
	
	\author{Mengni Li}
    \address{School of Mathematics, Southeast University, Nanjing 211189, P.R. China}
    \email{lmn@seu.edu.cn}
	
	\thanks{This work was supported by the National Natural Science Foundation of China (No. 11771214) and the Natural Science Foundation of Jiangsu Province (No. BK20220792).}

	
	\date{\today}
	\thanks{*corresponding author}
	

	\begin{abstract}
	In this paper, we consider a kind of $k$-Hessian type equations $S_k^{\frac{1}{k}}(D^2u+\mu|D u|I)= f(u)$ in $\mathbb{R}^n$, and provide a necessary and sufficient condition of $f$ on the existence and nonexistence of entire admissible subsolutions, which can be regarded as a generalized Keller-Osserman condition. The existence and nonexistence results are proved in different ranges of the parameter $\mu$ respectively, which embrace the standard Hessian equation case ($\mu=0$) by Ji and Bao (Proc Amer Math Soc 138: 175--188, 2010) as a typical example. The difference between the semilinear case ($k=1$) and the fully nonlinear case ($k\ge 2$) is also concerned.
\end{abstract}

\maketitle

\section{Introduction}\label{Section 1}
In this paper, we study a kind of $k$-Hessian type partial differential equations of the form,
\begin{equation}\label{1.0}
	S_k\left[ D^2u-A(x,u,Du)\right] = B(x,u,Du) ,\quad {\rm in}\ \mathbb{R}^n,
\end{equation}
where $Du$ and $D^2u$ denote the gradient vector and the Hessian matrix of $u$, $A$ is a $n\times n$ symmetric matrix function defined on $\mathbb{R}^n\times\mathbb{R}\times \mathbb{R}^n$, and $B:\mathbb{R}^n\times\mathbb{R}\times \mathbb{R}^n \to \mathbb{R}$ is a scalar valued function.  
For $k=1,\cdots,n$, the $k$-Hessian operator $S_k$ denotes the $k$-th elementary symmetric function defined by
\begin{equation*}
	S_k[r]:=S_{k}\left(\lambda(r)\right)  =\sum_{1\le i_{1}<\cdots< i_{k}\le n}\lambda_{i_{1}}\cdots\lambda_{i_{k}},
\end{equation*}
where $\lambda=(\lambda_1,\cdots,\lambda_n)$ denote the eigenvalues of the matrix $r$.

The classical model of \eqref{1.0} with $A=0$ has been investigated in the pioneering paper of Caffarelli, Nirenberg and Spruck \cite{CNS}. 
For extensive studies and outstanding results on the standard Hessian equation $S_k(D^2u)=B(x,u,Du)$, see, for example, \cite{CW,T,TW} and the references therein.
If $A\neq 0$, for $k=n$, equation \eqref{1.0} corresponds to a class of Monge-Amp\`{e}re type equations arising in optimal transportation \cite{MTW} and geometrical optics problems \cite{GH,W1}, which have been studied in \cite{LTW,JTY,HJL}. 	
For general $k$, if $A$ depends on $x$, a class of $k$-Hessian type equations \eqref{1.0} on Riemannian manifolds are studied by Urbas \cite{U}, Guan \cite{G,G2}, and others; 
if $A$ depends on $Du$, the $k$-Hessian type equations \eqref{1.0} is related to conformally invariant elliptic equations \cite{LL} and the $k$-Yamabe problems in conformal geometry \cite{V}; 
if $A$ depends on $x$, $u$, $Du$, the global regularity for Dirichlet and oblique boundary value problems of augmented Hessian equation \eqref{1.0} are studied in \cite{JT1,JT2}. 

More specifically, when $k=1$, $A=0$ and $B(x,u,Du)=f(u)$, \eqref{1.0} reduces to the classical Poisson equation
\begin{equation}\label{1.2}
	\Delta u=f(u),\quad {\rm in}\ \mathbb{R}^n.
\end{equation}
It is worthwhile to point out that Keller \cite{K} and Osserman \cite{O} simultaneously and independently proved that when $f$ is a positive, nondecreasing and continuous function, equation \eqref{1.2} admits a subsolution which exists in the whole space $\mathbb{R}^n$ if and only if 
\begin{equation}\label{1.3}
	\int^{+\infty}\left(\int_{0}^{\tau}f(t)\,{\rm d}t \right) ^{-\frac{1}{2}}\,{\rm d}\tau=+\infty,
\end{equation} 
where and later we omit the lower limit of integral to admit an arbitrary positive number, and the definition of subsolutions will be introduced later in this section.  
The growth condition on $f$ at infinity \eqref{1.3} is the celebrated Keller-Osserman condition.
For general $k$, Ji and Bao \cite{JB} extended the classical Keller-Osserman condition \eqref{1.3} to the standard Hessian equation
\begin{equation}\label{1.4}
	S_k^\frac{1}{k}(D^2u)=f(u)\quad {\rm in}\ \mathbb{R}^n,
\end{equation}
where $S_k^\frac{1}{k}$ denotes the $k$-th root of $S_k$.
Specifically, for $2\leq k<n$ and $k=n$, the generalized Keller-Osserman conditions for the solvability of subsolutions to \eqref{1.0} with $A=-\alpha I$ and $B=f^k(u)$, which include \eqref{1.4} as a very special case, are established in \cite{D1,D2} respectively.
For the generalized Keller-Osserman condition on $k$-Yamabe type equation, one can see \cite{BJL}.
In the fully nonlinear framework, if we replace the $k$-Hessian operator of \eqref{1.0} with more general second-order degenerate elliptic operators, analogous results have been more recently obtained in \cite{CDLV1,CDLV3}.
Remarkably, condition \eqref{1.3} is often used to study the boundary blow-up problems, see for example \cite{J,PV,S}. 

As far as $k$-Hessian type equations with gradient terms are concerned, in the case of $k=1$, some relevant research studied the following boundary blow-up problem
\begin{equation}\label{1.5}
	\left\{
	\begin{aligned}
		&\Delta u\pm g(|D u|)=f(u), \quad {\rm in}\ \Omega,\\
		&u\to +\infty,\quad \ {\rm dist}(x,\partial\Omega)\to 0 ,
	\end{aligned}
	\right.
\end{equation}
which appears in stochastic control problems with state constraints.  
Lasry and Lions \cite{LLions} first considered the case of $\pm g(|D u|)=-|D u|^p$ and $f(u)=\lambda u+h$, where $p>1$, $\lambda>0$ and $h$ is a given smooth function in $\Omega$. 
In this case, the equation in \eqref{1.5} is a very particular case of the so-called Hamilton-Jacobi-Bellman equations.
Subsequently, for $g(|D u|)=|D u|^q$, $q>0$, problem \eqref{1.5} was considered  by Bandle and Giarrusso in  \cite{BG}.
For general case, the existence and nonexistence of nonnegative solutions to problem \eqref{1.5} are proved in \cite{SMA}, depending on new integral conditions of Keller-Osserman type involving $f$ and $g$.
More interesting to us is that based on the degenerate maximal Pucci operator $\mathcal{M}_{0,1}^+$ and the $k$-partial Laplacian operator $\mathcal{P}_k^{+}$, Capuzzo Dolcetta, Leoni and Vitolo \cite{CDLV2} studied the more general fully nonlinear degenerate elliptic inequality with gradient term
\begin{equation}\label{1.6}
	F(x,D^2u)\geq f(u)+g(u)|Du|^q,\quad x\in \mathbb{R}^n,
\end{equation} 
for $q\in (0,2]$, and gave a necessary and sufficient condition for the solvability of \eqref{1.6} which extended the classical Keller-Osserman condition to semilinear equations.
We note that in recent research paper involving $k$-Hessian operator with gradient terms, Cui \cite{C} proved the existence of entire $k$-convex radial solutions to the Dirichlet problem of $k$-Hessian system
\begin{equation*}
	\left\{
	\begin{aligned}
		&S_k(\lambda(D^2u+\mu| D u|I))=p(|x|)f_1(u)f_2(v),\quad x\in B_1(0),\\
		&S_l(\lambda(D^2v+\mu| D v|I))=q(|x|)g_1(u)g_2(v),\quad \ x\in B_1(0),\\
		&u=v=0,\qquad \qquad\qquad\qquad\qquad\qquad\qquad \ x\in \partial B_1(0),
	\end{aligned}
	\right.
\end{equation*}
by using the monotone iterative method. For related studies of $k$-Hessian systems with $\mu=0$, see, for instance, \cite{ZZ,ZQ,JJD}.

In this paper, we are interested in studying the following $k$-Hessian type equation
\begin{equation}\label{1.1}
	S_k^{\frac{1}{k}}(D^2u+\mu|D u|I)=f(u),\quad {\rm in}\ \mathbb{R}^n,
\end{equation}
where $\mu$ is a constant,  $|D u|$ denotes the norm of gradient
vector of $u$, $I$ is the $n\times n$ identity matrix and $f$ is a given function defined on  $\mathbb{R}$. 
When $\mu=0$, equation \eqref{1.1} is the standard $k$-Hessian equation \eqref{1.4}, 
and when $\mu <0$, equation \eqref{1.1} satisfies the well-known Ma-Trudinger-Wang condition in \cite{MTW} at the points $D u\neq 0$. In fact, when $A$ in \eqref{1.0} depends on $D u$, the Heinz-Lewy example in \cite{Sch} shows that one can not expect $C^1$ regularity for solution $u$ in general. Even for the simpler equation \eqref{1.1}, it is not a trivial problem towards the existence of smooth solutions. Therefore, we only focus on this simpler case and discuss the entire smooth subsolutions of \eqref{1.1} in this paper.

To work in the realm of elliptic equations, we let 
\begin{equation*}
	\Gamma_{k}=\left\lbrace \lambda\in\mathbb{R}^{n}:S_{p}(\lambda)>0, \ \forall p=1,\cdots,k\right\rbrace
\end{equation*} 
be a convex cone and we always 
assume that $\lambda(D^2u+\mu|D u|I)\in \Gamma_{k}$.
Here, a function $u\in C^2(\Omega)$ is called admissible in $\Omega$ if $\lambda(D^2u+\mu|D u|I)\in \Gamma_{k}$ for all $x\in \Omega$. 
Let $\Phi_{\mu}^k$ denote the class of admissible functions in $\Omega$, i.e.,
\begin{equation*}
	\Phi_{\mu}^k(\Omega):=\{u\in C^2(\Omega):\lambda(D^2u+\mu|D u|I)\in \Gamma_{k}, \ \forall x\in \Omega\}.
\end{equation*}
A function $u\in C^2(\mathbb{R}^{n})$ is said to be an entire admissible subsolution of \eqref{1.1}, if $u\in \Phi_{\mu}^k(\mathbb{R}^n)$  and $u$ satisfies the following inequality
\begin{equation}\label{subsol}
	S_k^{\frac{1}{k}}(D^2u+\mu|D u|I)\geq f(u),\quad {\rm in}\ \mathbb{R}^n.
\end{equation}

\begin{remark}
	Specially, we note that if $k=1$, equation \eqref{1.1} reduces to the semilinear equation
	\begin{equation*}
		\Delta u+n \mu |D u|=f(u),\quad {\rm in}\ \mathbb{R}^n,
	\end{equation*}
	which is the special case of \eqref{1.5} with $\pm g(|D u|)=n \mu |D u|$ and is also the special case of \eqref{1.6} with $F=\Delta$, $g(u)=-n\mu $ and $q=1$. 
\end{remark}

In this paper, we extend the results in \cite{JB} to a kind of $k$-Hessian type equations with gradient terms. Letting 
\begin{equation}\label{mu 0}
\mu_0:=\sqrt{\frac{k}{n(k+1)(C_n^k)^{\frac{1}{k}}}},
\end{equation}
then our main results are: 
	\begin{Theorem}\label{Th1.1}
	Suppose that $f\in C(\mathbb{R})$ is a positive and monotone non-decreasing function. 
	\begin{enumerate}[(i)]
		\item (Existence) For either the case $k=1$, $\mu\in(-\infty,+\infty)$; or the case $2\leq k\leq n$, $\mu \in [0,+\infty)$, if $f$ satisfies
		\begin{equation}\label{nsc}
			\int^{+\infty}\left(\int_{0}^{\tau}f^k(t)\,{\rm d}t \right) ^{-\frac{1}{k+1}}\,{\rm d}\tau=+\infty,
		\end{equation}
		then equation \eqref{1.1} admits a subsolution $u\in \Phi_{\mu}^k(\mathbb{R}^n)$.
		\item (Nonexistence) For $1\leq k\leq n$, $\mu \in (-\infty,\mu_0)$, if $f$ satisfies
		\begin{equation*}
			\int^{+\infty}\left(\int_{0}^{\tau}f^k(t)\,{\rm d}t \right) ^{-\frac{1}{k+1}}\,{\rm d}\tau<+\infty,
		\end{equation*} 
		then equation \eqref{1.1} admits no subsolution $u\in \Phi_{\mu}^k(\mathbb{R}^n)$.
	\end{enumerate}
	\end{Theorem}

	\begin{Theorem}\label{Th1.2}
	Suppose that $f\in C(\mathbb{R})$ is a positive and monotone non-decreasing function. 
	For either the case $k=1$, $\mu\in(-\infty,\mu_0)$; or the case $2\leq k\leq n$, $\mu \in [0,\mu_0)$, 
	equation \eqref{1.1} admits a subsolution $u\in \Phi_{\mu}^k(\mathbb{R}^n)$ if and only if condition \eqref{nsc} holds.
	\end{Theorem}

Here, we obtain the existence and nonexistence results of entire subsolutions to equation \eqref{1.1} in Theorem \ref{Th1.1}. 
In Theorem \ref{Th1.2}, after adding an upper bound $\mu_0$ on $\mu$, we actually obtain a necessary and sufficient condition on the solvability for the subsolutions of equation \eqref{1.1}, which can be regarded as a generalized Keller-Osserman condition.
This boundness condition on $\mu$ comes from solving technical difficulties and plays an important role in the proof of the necessity.

\begin{remark}
	We point out that $[0, \mu_0)$ can be regarded as an interval to guarantee the Keller-Osserman condition \eqref{nsc} to be a necessary and sufficient condition of the global solvability of  inequality \eqref{subsol} when $k\ge 2$; and $(-\infty, \mu_0)$ can be regarded as an interval to guarantee the Keller-Osserman condition \eqref{nsc} to be a necessary and sufficient condition of the global solvability of inequality \eqref{subsol} when $k=1$. Namely that we have found the intervals for $\mu$ where the Keller-Osserman condition \eqref{nsc} could be a necessary and sufficient condition of the global solvability of inequality \eqref{subsol}.
\end{remark}

Note that the subsolution in Theorem \ref{Th1.1} and Theorem \ref{Th1.2} has no sign restriction. 
However, we are also interested in the existence of the positive subsolutions. 
If we extend the requirement of $f$ to the degenerate case, then the result of the positive subsolutions to equation \eqref{1.1} can be formulated as follows.

\begin{Theorem}\label{Th1.3}
	Suppose that $f\in C(\mathbb{R})$ satisfies 
	\begin{equation}\label{f0}
		\left\{
		\begin{aligned}
			&f(t) > 0 \ \text{is monotonically non-decreasing in}\ (0,+\infty),	\\
			&f(t) = 0 \ in\ (-\infty,0].
		\end{aligned}
		\right.
	\end{equation}
	For either the case $k=1$, $\mu\in(-\infty,\mu_0)$; or the case $2\leq k\leq n$, $\mu \in [0,\mu_0)$, equation \eqref{1.1} admits a positive subsolution $u\in \Phi_{\mu}^k(\mathbb{R}^n)$ if and only if condition \eqref{nsc} holds.
\end{Theorem}

Comparing with the standard Hessian equation in \cite{JB}, we make effort to deal with the gradient term $\mu |Du|I$ in this paper. We observe that the radial solution of equation \eqref{1.1} satisfies an ordinary differential equation with a divergence structure, see \eqref{integrating} for example. When $\mu<0$, different results are found in the semilinear case ($k=1$) and in the fully nonlinear case ($k\ge 2$), which is one contribution of this paper. Moreover, an upper bound $\mu_0$ of $\mu$ has been found so that $(-\infty, \mu_0)$ in the $k=1$ case and $[0, \mu_0)$ in the $k\ge 2$ case provide ranges where Keller-Osserman condition \eqref{nsc} becomes a necessary and sufficient condition towards the entire existence of inequality \eqref{subsol}. This is another contribution of this paper.

The rest of the paper is organized as follows. 
In Section \ref{Section 2}, we make some preliminary calculations of $C^2$ radial solutions and relate the $k$-Hessian type equation \eqref{1.1} to a Cauchy problem in the radial sense.
The proofs of the main theorems are given in Section \ref{Section 3}.
In Section \ref{Section 4}, we give two corollaries which are the applications of Theorems \ref{Th1.2} and \ref{Th1.3} with special forms of $f(u)$. Two examples of explicit entire positive subsolutions are presented at the end.

\vspace{3mm}

\section{Preliminaries}\label{Section 2}
In this section, we state some preliminary results of the radial functions.
We first calculate the eigenvalues of $k$-Hessian type operator in the radial sense and obtain the expression of $k$-Hessian type operator $S_k(D^2u+\mu |D u|I)$ in the radial sense in Lemma \ref{Lem2.1}.
Thus, we get an ordinary differential equation associated with equation \eqref{1.1}, which is equipped with a divergence structure. By integrating the ordinary differential equation once and imposing an initial condition, we get a Cauchy problem \eqref{Cauchy problem}.
Then in Lemma \ref{Lem2.2}, we prove the {\it a priori} properties of the solutions to the Cauchy problem.
Finally, the local existence result of the Cauchy problem is established in Lemma \ref{Lem2.3}.

Letting $r=|x|=\sqrt{\sum_{i=1}^{n}x_{i}^{2}}$, we formulate these preliminary lemmas as follows.
\begin{Lemma}\label{Lem2.1}
	Assume $\varphi(r)\in C^2[0,+\infty)$ is radially symmetric with $\varphi^{\prime}(0)=0$. Then for $u(x)=\varphi(r)$, we have that $u(x)\in C^2(\mathbb{R}^n)$ and the eigenvalues of  $\lambda(D^2u+\mu|D u|I)$ are
	\begin{equation}\label{lambda}
		\lambda(D^2u+\mu|D u|I)=
		\left\{
		\begin{aligned}
			&\left( \varphi^{\prime\prime}(r)+\mu\varphi^{\prime}(r),
			\frac{1+\mu r}{r}\varphi^{\prime}(r),
			\cdots,
			\frac{1+\mu r}{r}\varphi^{\prime}(r) \right),\quad
			r\in (0,+\infty),\\
			&\left( \varphi^{\prime\prime}(0),
			\varphi^{\prime\prime}(0),
			\cdots,
			\varphi^{\prime\prime}(0)\right),\quad r=0.
		\end{aligned}
		\right.
	\end{equation}
	Moreover,
	\begin{equation}\label{Sk}
		S_k(D^2u+\mu|D u|I)=
		\left\{
		\begin{aligned}
			&C_{n-1}^{k-1}
			(\varphi^{\prime\prime}(r)+\mu\varphi^{\prime}(r) ) 
			(\frac{1+\mu r}{r}\varphi^{\prime}(r) )^{k-1}+
			C_{n-1}^k
			(\frac{1+\mu r}{r}\varphi^{\prime}(r) )^{k},\quad
			r\in (0,+\infty),\\
			&C_n^k
			\left( \varphi^{\prime\prime}(0)\right)^k ,\quad r=0.
		\end{aligned}
		\right.
	\end{equation}
\end{Lemma}
\begin{proof}
	It is easy to see that for $x\neq 0$, $i,j=1,\cdots,n$, 
	\begin{equation}\label{d1}
		\frac{\partial u}{\partial x_i}(x)=\varphi^{\prime}(r) \frac{x_i}{r},
	\end{equation}
	and 
	\begin{equation}\label{d2}
		\frac{\partial ^2u}{\partial x_i \partial x_j}(x)
		=\left(  \frac{\varphi^{\prime\prime}(r)}{r^2} 
		-\frac{\varphi^{\prime}(r)}{r^3}     \right)x_i x_j  
		+\frac{\varphi^{\prime}(r)}{r} \delta_{ij}.
	\end{equation}
	Using \eqref{d1} and $\varphi^{\prime}(0)=0$, we have
	\begin{equation*}
		\lim_{x\to 0}\frac{\partial u}{\partial x_i}(x)=\lim_{x\to 0} \frac{\varphi^{\prime}(r)-\varphi^{\prime}(0)}{r-0}x_i=\varphi^{\prime\prime}(0)\cdot0=0.
	\end{equation*} 
	Similarly, by \eqref{d2}, we have
	\begin{equation}\label{d20}
		\lim_{x\to 0}\frac{\partial ^2u}{\partial x_i \partial x_j}(x)
		=\lim_{x\to 0} \left( \left(  \frac{\varphi^{\prime\prime}(r)}{r^2} 
		-\frac{\varphi^{\prime}(r)}{r^3}     \right)x_i x_j  
		+\frac{\varphi^{\prime}(r)}{r} \delta_{ij}\right) =\varphi^{\prime\prime}(0)\delta_{ij}.
	\end{equation} 
	Then $u(x)\in C^2(\mathbb{R}^n)$ is obvious by defining
	\begin{equation*}
		\frac{\partial u}{\partial x_i}(0)=0,\quad {\rm and} \ \ \frac{\partial ^2u}{\partial x_i \partial x_j}(0)=\varphi^{\prime\prime}(0)\delta_{ij}.
	\end{equation*}
	Combining \eqref{d2}, \eqref{d20} and
	\begin{equation*}
		|D u|= \left[ \sum_{i=1}^{n} \left( \frac{\partial u}{\partial x_i}\right) ^2 \right] ^\frac{1}{2}= \varphi^{\prime}(r)\left[ \sum_{i=1}^{n} \left( \frac{x_i}{r} \right) ^2\right]^\frac{1}{2} =\varphi^{\prime}(r),\quad r\geq0,
	\end{equation*} 
	then $D^2u+\mu|D u|I$ can be denoted as 
	\begin{equation}\label{ab}
		D^2u+\mu|D u|I=ax^Tx+bI,
	\end{equation}
	where
	\begin{equation*}
		a=
		\left\{
		\begin{aligned}
			& \frac{\varphi^{\prime\prime}(r)}{r^2} -\frac{\varphi^{\prime}(r)}{r^3}, & r\in (0,+\infty),\\
			&0, & r=0, \quad \quad\quad
		\end{aligned}
		\right.
	\end{equation*} 
	\begin{equation*}
		b=
		\left\{
		\begin{aligned}
			& \frac{\varphi^{\prime}(r)}{r} +\mu\varphi^{\prime}(r), 
			\quad r\in (0,+\infty),\\
			&\varphi^{\prime\prime}(0)+\mu\varphi^{\prime}(r) ,\quad  r=0.
		\end{aligned}
		\right.
	\end{equation*}
	Since the eigenvalues of matrix $x^Tx$ are $\lambda (x^Tx)=(|x|^2,0,\cdots,0)$, and so the eigenvalues of $ax^Tx+bI$ in \eqref{ab} are $(a|x|^2+b,b,\cdots,b)$. Hence \eqref{lambda} is proved. The proof of \eqref{Sk} is obtained by the definition of $S_k$.
\end{proof}

From Lemma \ref{Lem2.1}, we know that $u(x) = \varphi(r)\in C^2$ is a
radial solution of equation \eqref{1.1} if and only if $\varphi(r)$
satisfies the ordinary differential equation
\begin{equation}\label{ODE}
	C_{n-1}^{k-1}
	\left(\varphi^{\prime\prime}(r)+\mu\varphi^{\prime}(r) \right) 
	\left(\frac{1+\mu r}{r}\varphi^{\prime}(r) \right)^{k-1}+
	C_{n-1}^k
	\left(\frac{1+\mu r}{r}\varphi^{\prime}(r) \right)^{k}=f^k(\varphi(r))
\end{equation}
for $r\in(0,+\infty)$. By a simple calculation of \eqref{ODE}, we have
\begin{equation}\label{2.8}
	\begin{split}
		k\varphi^{\prime\prime}(r)\left( \varphi^{\prime}(r)\right) ^{k-1}+\left( n\mu+\frac{n-k}{r}\right) \left( \varphi^{\prime}(r)\right) ^{k}=\frac{k}{C_{n-1}^{k-1}}
		\left(\frac{r}{1+\mu r} \right)^{k-1}f^k(\varphi(r)).
	\end{split}
\end{equation}
Constructing a function $\chi(r)$ such that $\chi^{\prime}(r)=n\mu+\frac{n-k}{r}$, then we get $\chi(r)=n\mu r+(n-k)\ln r$.
Multiplying $e^{\chi(r)}$ on both sides of \eqref{2.8}, we obtain
\begin{equation}\label{integrating}
	\begin{split}
		\left[ \left( \varphi^{\prime}(r)\right) ^{k}e^{\chi(r)}\right] ^{\prime}
		&=k\varphi^{\prime\prime}(r)\left( \varphi^{\prime}(r)\right) ^{k-1}e^{\chi(r)}+\left( n\mu+\frac{n-k}{r}\right) \left( \varphi^{\prime}(r)\right) ^{k} e^{\chi(r)}\\
		&=\frac{k}{C_{n-1}^{k-1}}e^{\chi(r)}
		\left(\frac{r}{1+\mu r} \right)^{k-1}f^k(\varphi(r)).
	\end{split}
\end{equation}
Integrating \eqref{integrating} from $0$ to $r$ gives
\begin{equation*}
	\left( \varphi^{\prime}(r)\right) ^{k}e^{\chi(r)}=\int_{0}^{r} \frac{k}{C_{n-1}^{k-1}}e^{\chi(s)}
	\left(\frac{s}{1+\mu s} \right)^{k-1}f^k(\varphi(s)) \,{\rm d}s
\end{equation*}
for $r>0$, where $\varphi^{\prime}(0)=0$ is used.  Then we get the following expression of the derivative of the radial solution:
\begin{equation*}
	\varphi^{\prime}(r)= \left[ e^{-\chi(r)}
	\int_{0}^{r} \frac{k}{C_{n-1}^{k-1}}e^{\chi(s)}
	\left(\frac{s}{1+\mu s} \right)^{k-1}f^k(\varphi(s)) \,{\rm d}s\right] ^{\frac{1}{k}},
\end{equation*}
where $\chi(r):=n\mu r+(n-k)\ln r$.
For any given initial value $a$, we are going to consider the related Cauchy problem:
\begin{equation}\label{Cauchy problem}
	\left\{
	\begin{aligned}
		&\varphi^{\prime}(r)=
		\left(
		\frac{r^{k-n}}{e^{n\mu r } }
		\int_{0}^{r} \frac{k}{C_{n-1}^{k-1}}
		\frac{e^{n\mu s}s^{n-1}}{(1+\mu s)^{k-1}} 
		f^k(\varphi(s)) \,{\rm d}s
		\right)^{\frac{1}{k}},\quad r>0,\\
		&\varphi(0)=a.
	\end{aligned}
	\right.
\end{equation}

\begin{Lemma}\label{Lem2.2}
	Let $f(t)>0$ be a continuous and monotone non-decreasing function defined on $\mathbb{R}$. 
	For any constant $a$, we assume that there exists $\varphi(r) \in C^0[0,+\infty)\cap C^1(0,+\infty)$ satisfying the Cauchy problem \eqref{Cauchy problem}.
	Then $\varphi(r)\in C^2[0,+\infty)$, and it satisfies equation \eqref{ODE} with $\varphi^{\prime}(0)=0$. 
	Furthermore, for either the case $k=1$, $\mu\in(-\infty,+\infty)$; or the case $2\leq k\leq n$, $\mu \in [0,+\infty)$, we have
	\begin{equation}\label{k-cone}
		\lambda_r:=\left( \varphi^{\prime\prime}(r)+\mu\varphi^{\prime}(r),
		\frac{1+\mu r}{r}\varphi^{\prime}(r),
		\cdots,
		\frac{1+\mu r}{r}\varphi^{\prime}(r) \right)\in \Gamma_{k}, \quad r\in (0,+\infty).
	\end{equation}
\end{Lemma}
\begin{proof}
	We first prove $\varphi(r)\in C^2[0,+\infty)$. It is easy to see that $\varphi(r)\in C^2(0,+\infty)$.
	By using the L'Hospital rule,  we have 
	\begin{equation*}
		\begin{split}
			\lim_{r\to 0}\varphi^{\prime}(r)
			&=\lim_{r\to 0} \left(
			\frac{r^{k-n}}{e^{n\mu r } }
			\int_{0}^{r} \frac{k}{C_{n-1}^{k-1}}
			\frac{e^{n\mu s}s^{n-1}}{(1+\mu s)^{k-1}} 
			f^k(\varphi(s)) \,{\rm d}s
			\right)^{\frac{1}{k}}\\
			&=\left(  \frac{k}{C_{n-1}^{k-1}}	\right) ^{\frac{1}{k}}
			\left(  \lim_{r\to 0}
			\frac{\int_{0}^{r} e^{n\mu s}
				\frac{s^{n-1}}{(1+\mu s)^{k-1}} 
				f^k(\varphi(s)) \,{\rm d}s}
			{e^{n\mu r}r^{n-k}} \right) ^{\frac{1}{k}}\\
			&=\left(  \frac{k}{C_{n-1}^{k-1}}	\right) ^{\frac{1}{k}}
			\left(  \lim_{r\to 0}	
			\frac{e^{n\mu r} 
				\frac{r^{n-1}}{(1+\mu r)^{k-1}}
				f^k(\varphi(r))}
			{n\mu e^{n\mu r }r^{n-k}+e^{n\mu r}(n-k)r^{n-k-1}} 
			\right) ^{\frac{1}{k}}\\
			&=\left(  \frac{k}{C_{n-1}^{k-1}}	\right) ^{\frac{1}{k}}
			\left( \lim_{r\to 0}  	
			\frac{f^k(\varphi(r))}
			{\left( n\mu +(n-k)\frac{1}{r}\right)
				\left( \frac{1}{r}+\mu\right) ^{k-1}} 
			\right) ^{\frac{1}{k}}
			=0=\varphi^{\prime}(0).
		\end{split}
	\end{equation*}
	Then $\varphi(r)\in C^1[0,+\infty)$.
	Differentiating \eqref{Cauchy problem}  with respect to $r$, we have for $r>0$,
	\begin{equation}\label{f2}
		\begin{split}
			\varphi^{\prime\prime}(r) 
			=&\frac{1}{k} \left( \frac{k}{C_{n-1}^{k-1}}   \right) ^{\frac{1}{k}} 
			\left[ 
			\left((k-n)\frac{1}{r}- n\mu   \right) 
			\left(\frac{r^{k-n}}{e^{n\mu r}}
			\int_{0}^{r}
			\frac{e^{n\mu s}s^{n-1}}{(1+\mu s)^{k-1}}
			f^k(\varphi(s)) \,{\rm d}s \right) ^{\frac{1}{k}}  
			\right.\\
			&\left. 
			+\left( \frac{r}{1+\mu r} \right) ^{k-1}
			f^k(\varphi(r))
			\left(\frac{r^{k-n}}{e^{n\mu r}}
			\int_{0}^{r}
			\frac{e^{n\mu s}s^{n-1}}{(1+\mu s)^{k-1}}
			f^k(\varphi(s)) \,{\rm d}s \right) ^{\frac{1}{k}-1} 
			\right].
		\end{split}	
	\end{equation}
	On the one hand, we take a limit of \eqref{f2} at $0$ to have 
	\begin{equation*}
		\begin{split}
			&\lim_{r\to 0} \varphi^{\prime\prime}(r)\\
			=&\lim_{r\to 0}  
			\frac{1}{k} \left( \frac{k}{C_{n-1}^{k-1}}   \right) ^{\frac{1}{k}} 
			\left[ 
			\left((k-n)\frac{1}{r}- n\mu   \right) 
			\left(\frac{r^{k-n}}{e^{n\mu r}}
			\int_{0}^{r}
			e^{n\mu s} \frac{s^{n-1}}{(1+\mu s)^{k-1}}
			f^k(\varphi(s)) \,{\rm d}s \right) ^{\frac{1}{k}}  
			\right.\\
			&\left. 
			+\left( \frac{r}{1+\mu r} \right) ^{k-1}
			f^k(\varphi(r))
			\left(\frac{r^{k-n}}{e^{n\mu r}}
			\int_{0}^{r}
			e^{n\mu s} \frac{s^{n-1}}{(1+\mu s)^{k-1}}
			f^k(\varphi(s)) \,{\rm d}s \right) ^{\frac{1}{k}-1} 
			\right]\\
			=&\frac{1}{k} \left( \frac{k}{C_{n-1}^{k-1}}   \right) ^{\frac{1}{k}} 
			\left[ 
			(k-n) \left( \lim_{r\to 0}  \frac{
				\frac{r^{n-1}}{(1+\mu r)^{k-1}}
				f^k(\varphi(r))    }{ nr^{n-1}+n \mu r^n}  \right) ^{\frac{1}{k}} 
			-n\mu \left(  \lim_{r\to 0}  \frac{
				\frac{r^{n-1}}{(1+\mu r)^{k-1}}
				f^k(\varphi(r))    }{(n-k)r^{n-k-1}+n \mu r^{n-k}}  \right) ^{\frac{1}{k}}  
			\right.\\
			&\left.
			+   f^k(\varphi(0)) \left(\lim_{r\to 0}  \frac{
				\frac{r^{n-1}}{(1+\mu r)^{k-1}}
				f^k(\varphi(r))    }
			{ \frac{k}{(1+\mu r)^2}  \frac{r^{n-1}}{(1+\mu r)^{k-1}}  +
				\frac{r^{n}}{(1+\mu r)^{k}} \left( (n-k)\frac{1}{r}+n\mu\right)  
			}  \right) ^{\frac{1}{k}-1}  
			\right] \\
			=&\frac{1}{k} \left( \frac{k}{C_{n-1}^{k-1}}   \right) ^{\frac{1}{k}} 
			\left[ 
			(k-n) \frac{f(\varphi(0))}{n^{\frac{1}{k}}}
			+   \frac{f(\varphi(0))}{n^{\frac{1}{k}-1}}\right] 
			=\frac{f(\varphi(0))}{(C_n^k)^{\frac{1}{k}}}.
		\end{split}
	\end{equation*}
	On the other hand, by the definition of the second derivative, we have 
	\begin{equation*}
		\begin{split}
			\varphi^{\prime\prime}(0)
			&=	\lim_{r\to 0}
			\frac{\left(
				\frac{r^{k-n}}{e^{n\mu r } }
				\int_{0}^{r} \frac{k}{C_{n-1}^{k-1}}e^{n\mu s}
				\frac{s^{n-1}}{(1+\mu s)^{k-1}} 
				f^k(\varphi(s)) \,{\rm d}s
				\right)^{\frac{1}{k}}-\varphi^{\prime}(0)}
			{r-0}\\
			&=\left( \frac{k}{C_{n-1}^{k-1}}   \right) ^{\frac{1}{k}} 
			\left( \lim_{r\to 0}
			\frac{\int_{0}^{r} e^{n\mu s}
				\frac{s^{n-1}}{(1+\mu s)^{k-1}} 
				f^k(\varphi(s)) \,{\rm d}s}
			{e^{n\mu r}r^{n}} \right) ^{\frac{1}{k}} \\
			&=\left( \frac{k}{C_{n-1}^{k-1}}   \right) ^{\frac{1}{k}} 
			\left( \lim_{r\to 0}
			\frac{f^k(\varphi(r))}
			{(n+n\mu r)(1+\mu r)^{k-1}}   
			\right)^\frac{1}{k} \\
			&=\left( \frac{k}{C_{n-1}^{k-1}}   \right) ^{\frac{1}{k}}  
			\frac{f(\varphi(0))}{n^{\frac{1}{k}}}
			=\frac{f(\varphi(0))}{(C_n^k)^{\frac{1}{k}}}.
		\end{split}
	\end{equation*}
	Thus $\lim_{r\to 0} \varphi^{\prime\prime}(r)=\varphi^{\prime\prime}(0)$, which shows that $\varphi(r)\in C^2[0,+\infty)$.
	By \eqref{Cauchy problem} and \eqref{f2}, it is easy to verify that $\varphi(r)$ satisfies equation \eqref{ODE}. 
	
	For $k=1$, it follows from \eqref{Sk} and \eqref{ODE} that
	\begin{equation*}
		S_1(D^2u+\mu|D u|I)=\varphi^{\prime\prime}(r)+(n\mu +\frac{n-1}{r})\varphi^{\prime}(r)=f(\varphi(r)).
	\end{equation*}
	Since $f$ is positive, we have $S_1(D^2u+\mu|D u|I)>0$, which implies that \eqref{k-cone} holds for $\mu\in (-\infty,+\infty)$.
	
	For $2\leq k\leq n$, 
	since $f>0$ and $\mu\geq 0$, it follows from \eqref{Cauchy problem} that $\varphi^{\prime}(r)>0$ for $r>0$, therefore $\frac{1+\mu r}{r}\varphi^{\prime}(r)>0$  for $r>0$.
	On the one hand, by \eqref{Sk} and \eqref{ODE}, we have 
	\begin{equation}\label{first line inequality}
		\begin{split}
			S_k(D^2u+\mu|D u|I)
			&=C_{n-1}^{k-1}
			\left( \varphi^{\prime\prime}(r)+\mu\varphi^{\prime}(r) \right)  
			\left( \frac{1+\mu r}{r}\varphi^{\prime}(r)\right) ^{k-1}+
			C_{n-1}^k
			\left( \frac{1+\mu r}{r}\varphi^{\prime}(r)\right) ^{k}\\
			&=f^k(\varphi(r)) >0.
		\end{split}
	\end{equation}
	On the other hand, it follows from \eqref{first line inequality} that 
	\begin{equation*}
		\begin{split}
			S_k(D^2u+\mu|D u|I)
			&=C_{n-1}^{k-1}
			\left( \frac{1+\mu r}{r}\varphi^{\prime}(r) \right) ^{k-1}
			\left[ \varphi^{\prime\prime}(r)+ \mu\varphi^{\prime}(r)
			+\frac{n-k}{k}\left( \frac{1+\mu r}{r}\varphi^{\prime}(r)\right) \right]\\
			&=C_{n-1}^{k-1}
			\left( \frac{1+\mu r}{r}\varphi^{\prime}(r) \right) ^{k-1}
			\left[ \varphi^{\prime\prime}(r)+\frac{n-k+n\mu r}{kr}\varphi^{\prime}(r)\right].
		\end{split}
	\end{equation*}
	To sum up, they show that 
	\begin{equation*}
		\varphi^{\prime\prime}(r)+\frac{n-k+n\mu r}{kr}\varphi^{\prime}(r)>0.
	\end{equation*}
	Then for $2\leq p \leq k$, we obtain
	\begin{equation*}
		\begin{split}
			S_p(D^2u+\mu|D u|I)
			&=C_{n-1}^{p-1}
			\left( \varphi^{\prime\prime}(r)+\mu\varphi^{\prime}(r) \right)  
			\left( \frac{1+\mu r}{r}\varphi^{\prime}(r)  \right) ^{p-1}+
			C_{n-1}^p
			\left( \frac{1+\mu r}{r}\varphi^{\prime}(r) \right) ^{p}\\
			&=C_{n-1}^{p-1}
			\left( \frac{1+\mu r}{r}\varphi^{\prime}(r) \right) ^{p-1}
			\left[ \varphi^{\prime\prime}(r)+\frac{n-p+n\mu r}{pr}\varphi^{\prime}(r)\right] \\
			&\geq C_{n-1}^{p-1}
			\left( \frac{1+\mu r}{r}\varphi^{\prime}(r) \right) ^{p-1}
			\left[ \varphi^{\prime\prime}(r)+\frac{n-k+n\mu r}{kr}\varphi^{\prime}(r)\right] >0.
		\end{split}
	\end{equation*}
	Hence \eqref{k-cone} is proved.
\end{proof}
\begin{remark}\label{re2.1}
	Note that for either the case $k=1$, $\mu\in(-\infty,+\infty)$; or the case $2\leq k\leq n$, $\mu \in [0,+\infty)$, if $f>0$, we have $\varphi^{\prime}(r)>0$ for $r>0$. 
	Moreover, for $2\leq k\leq n$,  $\mu<0$, we see that $\frac{1+\mu r}{r}\varphi^{\prime}(r) <0$ when $r>-\frac{1}{\mu}$. Then \eqref{k-cone} fails in this case, namely that we can not find an admissible solution defined on the whole space to the Cauchy problem \eqref{Cauchy problem}.
\end{remark}

Finally we shall use a method of Euler's break line to prove the local existence of the Cauchy problem \eqref{Cauchy problem} on some interval $[0, R)$, which is widely used in studying the existence theorem of classic ordinary differential equation and numerical methods for partial differential equations. The proof here is similar to Lemma 2.3 in \cite{JB}.
	\begin{Lemma}\label{Lem2.3}(Local existence)
	Let $f(t)>0$ be a continuous and monotone non-decreasing function defined on $\mathbb{R}$. 
	Then for any constant $a$, for either the case $k=1$, $\mu\in(-\infty,+\infty)$; or the case $2\leq k\leq n$, $\mu \in [0,+\infty)$, there exists a positive number $R$ such that the Cauchy problem \eqref{Cauchy problem} has a solution $\varphi \in C^2[0, R)$.
	\end{Lemma}
	\begin{proof}
	The proof is divided into three steps. We first define a Euler's break line and then show that this Euler's break line is an $\varepsilon$-approximation solution of \eqref{Cauchy problem}. In the third step, we find a solution of \eqref{Cauchy problem} by this Euler's break line.

	\textit{Step 1: Definition of Euler's break line $\psi$}.
	We define a functional $F(\cdot, \cdot)$ on
	\begin{equation*}
		\mathcal{R}:=[0,l]\times\{\varphi \in C^2[0,l]:a-h<\varphi <a+h\}
	\end{equation*}
	as 
	\begin{equation*}
		F(r,\varphi ):=
		\left(
		\frac{r^{k-n}}{e^{n\mu r } }
		\int_{0}^{r} \frac{k}{C_{n-1}^{k-1}}
		\frac{e^{n\mu s}s^{n-1}}{(1+\mu s)^{k-1}} 
		f^k(\varphi(s)) \,{\rm d}s
		\right)^{\frac{1}{k}},
	\end{equation*}	
	where $l$ and $h$ are small enough positive constants.
	Then 
	\begin{equation}
		\varphi^{\prime}(r)=F(r,\varphi )
	\end{equation}
	and $F(r,\varphi )>0$ for $r>0$.
	Define a Euler's break line on $[0,l]$ as 
	\begin{equation}\label{bl}
		\left\{
			\begin{aligned}
			&\psi(0)=a,\\
			&\psi(r)=\psi(r_{i-1})+F\left( r_{i-1},\psi(r_{i-1})\right) (r-r_{i-1}), \quad r_{i-1} < r < r_i,
			\end{aligned}
		\right.
	\end{equation}
	where $0=r_0<r_1<\cdots<r_m=l$. Then $\psi \in C^2[0,l]$. We claim that $(r,\psi)\in \mathcal{R}$.
	In fact, for any $(r,\varphi)\in \mathcal{R}$, the discussion on the following two cases will prove this claim.
	
	 For $1\leq k \leq n$, $\mu \geq 0$, since $f$ is monotonically non-decreasing, we have 
	\begin{equation}\label{F}
		\begin{split}
			F(r,\varphi)
			&\leq \left(
			\frac{r^{k-n}}{e^{n\mu r } }
			\int_{0}^{r} \frac{k}{C_{n-1}^{k-1}}
			\frac{e^{n\mu s}s^{n-1}}{(1+\mu s)^{k-1}}  \,{\rm d}s
			\right)^{\frac{1}{k}}  f(a+h)\\
			&\leq \left(
			\frac{k}{C_{n-1}^{k-1}}
	  		  \frac{r^k}{(1+\mu r)^{k-1}} 
			\right)^{\frac{1}{k}} f(a+h)\\
			&\leq \left(	\frac{k}{C_{n-1}^{k-1}} \right) ^{\frac{1}{k}} r f(a+h)\\
			&\leq \left(	\frac{k}{C_{n-1}^{k-1}} \right) ^{\frac{1}{k}} l f(a+h),
		\end{split}
	\end{equation}
	which implies that
	\begin{equation*}
		M:=\max_{\mathcal{R}} F(r,\varphi)\leq \left(	\frac{k}{C_{n-1}^{k-1}} \right) ^{\frac{1}{k}}l f(a+h).
	\end{equation*}
	Thus for the break line $(r,\psi)$, we have 
	\begin{equation}\label{M}
		\begin{split}
	    a-h<a \leq \psi(r)
		&\leq a+ \sum_{i} F(r_{i-1},\psi (r_{i-1})) (r-r_{i-1}) \\
		&\leq a+ M r\\
		&\leq a+\left(	\frac{k}{C_{n-1}^{k-1}} \right) ^{\frac{1}{k}}l^2 f(a+h), \quad r\in [0,l].
		\end{split}
	\end{equation}	
If $h$ is fixed, we can choose $l$ sufficiently small such that  $\psi (r)\in (a-h,a+h)$ for $r\in [0,l]$.

	For $k=1$, $\mu<0$, we have 
	\begin{equation}\label{F'}
		F(r,\varphi )=
		\frac{1}{e^{n\mu r } r^{n-1}}
		\int_{0}^{r} 
		e^{n\mu s}s^{n-1}
		f(\varphi(s)) \,{\rm d}s
		\leq \frac{r}{e^{n\mu r}} f(a+h)
		\leq \frac{l}{e^{n\mu l}} f(a+h).
	\end{equation}
	Letting
	\begin{equation*}
		M^{\prime }:=\max_{\mathcal{R}} F(r,\varphi)\leq \frac{l}{e^{n\mu l}} f(a+h),
	\end{equation*}
	then for the break line $(r,\psi)$, we have  
	\begin{equation}\label{M'}
		a-h<a\leq\psi(r)\leq a+M^{\prime}r\leq a+\frac{l^2}{e^{n\mu l}} f(a+h), \quad r\in [0,l].
	\end{equation}
If $h$ is fixed, we can choose $l$ sufficiently small such that  $\psi (r)\in (a-h,a+h)$ for $r\in [0,l]$.

	\textit{Step 2: We prove that the Euler's break line $\psi$ is an $\varepsilon$-approximation solution of \eqref{Cauchy problem}}.
	For this, we need to prove that for any $\varepsilon > 0$, there exists a sequence of appropriate points $\{r_i\}_{i=1,\cdots,m}$ such that the break line satisfies
	\begin{equation}\label{2.3g}
		\left| \frac{{\rm d}\psi(r)}{{\rm d}r}-F(r,\psi(r))\right| <\varepsilon, \quad r\in [0,l].
	\end{equation}
	For either the case $k=1$, $\mu\in(-\infty,+\infty)$; or the case $2\leq k\leq n$, $\mu \in [0,+\infty)$, from $(r,\psi)\in \mathcal{R}$ combining with the third line in \eqref{F} and \eqref{F'}, we have
	\begin{equation*}
		\lim_{r\to 0}F(r,\psi)=0
	\end{equation*}
	uniformly for $\psi \in C^2[0,l]$, $a-h<\psi<a+h$. Thus for any $\varepsilon>0$, there exists $\bar r\in (0,l)$ such that for $0\leq r_{i-1}< r< \bar r$,
	\begin{equation*}
		F(r,\psi(r))<\frac{\varepsilon}{2}, \quad
		F(r_{i-1},\psi(r_{i-1}))<\frac{\varepsilon}{2},
	\end{equation*} 
	and then 
	\begin{equation*}
		\left| \frac{{\rm d}\psi(r)}{{\rm d}r}-F(r,\psi(r))\right| 
		=\left| F(r_{i-1},\psi (r_{i-1}))-F(r,\psi(r))\right| <\varepsilon,
		\quad 0\leq r< \bar r.
	\end{equation*}
    For $\bar r\leq r\leq l$, we have
	\begin{equation*}
		\begin{split}
			&\left| \frac{{\rm d}\psi(r)}{{\rm d}r}-F(r,\psi(r))\right|\\
			=&\left| F(r_{i-1},\psi (r_{i-1}))-F(r,\psi(r))\right| \\
			=& \left(\frac{k}{C_{n-1}^{k-1}} \right) ^{\frac{1}{k}}
			\left|
			\left(
			\frac{r_{i-1}^{k-n}}{e^{n\mu r_{i-1} } }
			\int_{0}^{r_{i-1}}
			\frac{e^{n\mu s}s^{n-1}}{(1+\mu s)^{k-1}} 
			f^k(\varphi(s)) \,{\rm d}s
			\right)^{\frac{1}{k}}
			-
			\left(
			\frac{r^{k-n}}{e^{n\mu r } }
			\int_{0}^{r}
			\frac{e^{n\mu s}s^{n-1}}{(1+\mu s)^{k-1}} 
			f^k(\varphi(s)) \,{\rm d}s
			\right)^{\frac{1}{k}}
			\right|\\
			\leq &	\left(\frac{k}{C_{n-1}^{k-1}} \right) ^{\frac{1}{k}}
			\left| 
			\frac{r_{i-1}^{k-n}}{e^{n\mu r_{i-1} } }
			\int_{0}^{r_{i-1}}
			\frac{e^{n\mu s}s^{n-1}}{(1+\mu s)^{k-1}}
			f^k(\varphi(s)) \,{\rm d}s - 
	    	\frac{r^{k-n}}{e^{n\mu r } }
	   		 \int_{0}^{r}
	    	\frac{e^{n\mu s}s^{n-1}}{(1+\mu s)^{k-1}}
	  	 	 f^k(\varphi(s)) \,{\rm d}s
			\right| ^\frac{1}{k}\\
			= &\left(\frac{k}{C_{n-1}^{k-1}} \right) ^{\frac{1}{k}}
			\left( 
			\left| 
			\frac{r_{i-1}^{k-n}}{e^{n\mu r_{i-1} } }- \frac{r^{k-n}}{e^{n\mu r } }
			\right| 
			\int_{0}^{r_{i-1}}
			\frac{e^{n\mu s}s^{n-1}}{(1+\mu s)^{k-1}}
			f^k(\varphi(s)) \,{\rm d}s 
			+\frac{r^{k-n}}{e^{n\mu r } }
			\int_{r_{i-1}}^{r}\frac{e^{n\mu s}s^{n-1}}{(1+\mu s)^{k-1}}
			f^k(\varphi(s)) \,{\rm d}s
			\right)^\frac{1}{k} \\
			\leq &\left(\frac{k}{C_{n-1}^{k-1}} \right) ^{\frac{1}{k}}
			\left( 
			\left| 
			\frac{r_{i-1}^{k-n}}{e^{n\mu r_{i-1} } }- \frac{r^{k-n}}{e^{n\mu r } }
			\right| 
			\frac{e^{n\mu r}r^{n}}{(1+\mu r)^{k-1}}
			f^k(a+h) 
			+
			\frac{r^{k-n}}{e^{n\mu r } }
			\frac{e^{n\mu r}r^{n-1}}{(1+\mu r)^{k-1}}
			f^k(a+h) (r-r_{i-1})
			\right) ^\frac{1}{k}\\
			\leq &\left(\frac{k}{C_{n-1}^{k-1}} \right) ^{\frac{1}{k}}
			\left[		
			\left( 
	 		\left( e^{n\mu(r-r_{i-1})+(n-k)(\ln r-\ln r_{i-1})}-1 \right) l
			+(r-r_{i-1})\right)
			\left( \frac{l}{1+\mu l}\right) ^{k-1} f^k(a+h) 
			\right]  ^\frac{1}{k}.
		\end{split}
	\end{equation*}
	Let $\eta:=\frac{k}{C_{n-1}^{k-1}}\left( \frac{l}{1+\mu l}\right) ^{k-1} f^k(a+h)$. Since $\ln r$ is continuous on $[\bar r,l]$, then $\ln r$ is uniformly continuous on $[\bar r,l]$. For the above $\varepsilon>0$, for $r^{\prime},r^{\prime \prime }\in[\bar r ,l]$, we need to assume that $r^{\prime}$ and $r^{\prime\prime}$ are close in the sense that
	\begin{equation}\label{2.18}
		\begin{aligned}
		|r^{\prime}-r^{\prime\prime}|&<\frac{1}{2}\frac{\varepsilon^k}{\eta},\\
		|\ln r^{\prime}-\ln r^{\prime\prime}|
		&<\frac{1}{n-k}\left( \ln\left(1+\frac{1}{2l}\frac{\varepsilon^k}{\eta} \right)-\frac{n|\mu|}{2}\frac{\varepsilon^k}{\eta} \right).
		\end{aligned}
	\end{equation}
	Indeed, for $r^{\prime},r^{\prime \prime }\in[\bar r ,l]$, letting
	$$\delta(\varepsilon) =\min \left\{\frac{1}{2}\frac{\varepsilon^k}{\eta}, \frac{\bar r}{n-k}\left( \ln\left(1+\frac{1}{2l}\frac{\varepsilon^k}{\eta} \right)-\frac{n|\mu|}{2}\frac{\varepsilon^k}{\eta} \right)\right\},$$
	then the two inequalities in \eqref{2.18} are automatically satisfied by using the mean value theorem when $|r^{\prime}-r^{\prime\prime}|< \delta(\varepsilon)$.
	From \eqref{2.18}, we get
	\begin{equation*}
		\begin{aligned}
			\left( e^{n\mu(r^{\prime}-r^{\prime \prime })+(n-k)(\ln r^{\prime}-\ln r^{\prime \prime })}-1 \right) l
			\le \left( e^{n|\mu| |r^{\prime}-r^{\prime \prime }|+(n-k)|\ln r^{\prime}-\ln r^{\prime \prime }|}-1 \right) l<\frac{1}{2} \frac{\varepsilon^k}{\eta}.
		\end{aligned}
	\end{equation*}
	Combining this with \eqref{2.18}$_1$, we obtain
	\begin{equation*}
		\left( e^{n\mu(r^{\prime}-r^{\prime \prime })+(n-k)(\ln r^{\prime}-\ln 	r^{\prime \prime })}-1 \right) l+(r^{\prime}-r^{\prime \prime }) <  \frac{\varepsilon^k}{\eta},
	\end{equation*}
	which implies 
	\begin{equation*}
   		 \left(\frac{k}{C_{n-1}^{k-1}} \right) ^{\frac{1}{k}}
		\left[		
		\left( 
		\left( e^{n\mu(r^{\prime}-r^{\prime \prime })+(n-k)(\ln r^{\prime}-\ln 	r^{\prime \prime })}-1 \right) l
		+(r^{\prime}-r^{\prime \prime })\right)
		\left( \frac{l}{1+\mu l}\right) ^{k-1} f^k(a+h) 
		\right]  ^\frac{1}{k}<\varepsilon.
	\end{equation*}
	Noticing that  $\delta(\varepsilon)$ is independent of the definition of $\psi$, we can suppose that $r_1=\bar r $ and $\max_{2\leq i \leq m}|r_{i-1}-r_{i}|<\min\{\bar r,\delta(\varepsilon) \}$, then \eqref{2.3g} is obtained.
	Hence, the Euler's break line $\psi$ we defined in \eqref{bl} is an $\varepsilon$-approximation solution of \eqref{Cauchy problem}. 

	\textit{Step 3: We shall find a solution of \eqref{Cauchy problem} by the Euler's break line $\psi$}. 
	Let $\{\varepsilon_j\}_{j=1}^{\infty}$ be a sequence of positive constants converging to $0$. Assume that sequence $\{\psi_j\}$ is an $\varepsilon_j$-approximation solution of \eqref{Cauchy problem} on $[0,l]$, that is, 
	\begin{equation}\label{2.16}
		\left| \frac{{\rm d}\psi_j(r)}{{\rm d}r}-F(r,\psi_j(r))\right| <\varepsilon_j.
	\end{equation} 
	By \eqref{F}, we have
	\begin{equation*}
		\left| \psi_j(r^{\prime \prime })-\psi_j(r^{\prime})\right| <M |r^{\prime \prime }-r^{\prime}|
	\end{equation*}
	for  $r^{\prime \prime },r^{\prime}\in [0,l]$.
	It is easy to see that $\{\psi_j\}$ is equicontinuous and uniformly bounded.
	By the Ascoli-Arzela Lemma, there exists a uniformly convergent subsequence still denoted as $\{\psi_j\}$, without loss of generality. Assume $\lim _{j\to +\infty}\psi_j=\varphi$. Then $\varphi(0)=a$ and $\varphi^{\prime}(0)=0$. By \eqref{2.16}, we have 
	\begin{equation*}
		\frac{{\rm d}\psi_j(r)}{{\rm d}r}=F(r,\psi_j(r))+\triangle_j(r),
	\end{equation*}
	where $|\triangle_j(r)|<\varepsilon_j$ for $r\in[0,l]$. After integrating, we let $j\to +\infty$ and obtain
	\begin{equation}\label{2.17}
		\begin{split}
			\varphi(r)=\lim_{j\to +\infty}\psi_j(r)
			&=a+\lim_{j\to +\infty}\int_0^r\left[  F(s,\psi_j(s))+\triangle_j(s)\right] \,{\rm d}s\\
			&=a+\int_0^r \lim_{j\to +\infty}\left[  F(s,\psi_j(s))+\triangle_j(s)\right]\,{\rm d}s\\
			&=a+\int_0^r F(s,\varphi(s))\,{\rm d}s.
		\end{split}
	\end{equation}
	Differentiating \eqref{2.17}, we can see that $\varphi(r)$ satisfies \eqref{Cauchy problem} on $[0,l]$. Up to now, we can extend the interval $[0,l]$ to that with longer length, and thus there  exists a positive number $R$ such that $\varphi(r)$ satisfies \eqref{Cauchy problem} on $[0,R]$, especially on $[0,R)$.
	\end{proof}
\vspace{3mm}

\section{Proofs of main results}\label{Section 3}
In this section, we first prove Theorem \ref{Th1.2} by four steps. Then by modifying the conditions in the third step, we also get the proof of Theorem \ref{Th1.1}, which shows the results of existence and nonexistence. The degenerate case in Theorem \ref{Th1.3} is also proved.

First, we start to prove Theorem \ref{Th1.2} step by step.

	\begin{proof}[Proof of Theorem \ref{Th1.2}]
	The proof is divided into four steps. 
	In the first step, we discuss the relation between the admissible subsolutions to \eqref{1.1} and the solutions to \eqref{ODE} which satisfies \eqref{k-cone} in a bounded ball.
	Then we build an entire existential equivalence relationship between the admissible subsolutions to \eqref{1.1} and the admissible solutions to the Cauchy problem \eqref{Cauchy problem} in the second step. 
	In the third step, a necessary and sufficient condition for the existence and nonexistence of entire admissible solutions to the Cauchy problem \eqref{Cauchy problem} is proved.
	In the last step, we conclude that the necessary and sufficient condition shown in the third step is of the solvability for entire admissible subsolutions of equation \eqref{1.1}.

	\textit{Step 1:  $u(x)\leq\varphi(r)$ in $B_R(0)$}. 
	Suppose that there exists $\varphi(r)\in C^2[0,R)$ satisfying \eqref{ODE} and \eqref{k-cone} for $r\in [0,R)$, with $\varphi^{\prime}(0)=0$ and $\varphi(r)\to +\infty$ as $r\to R$. Meanwhile, $u(x)\in \Phi_\mu^k(\mathbb{R}^n)$ is a subsolution of equation \eqref{1.1}. 
	We let $v(x)=\varphi(r)$, where $r=|x|$. It follows from Lemma \ref{Lem2.2} that $\lambda(D^2v+\mu|D v|I)\in 
	\Gamma_k$, which means $v(x)\in \Phi_\mu^k(B_R(0))$ and $v(x)$ satisfies $S_k(D^2v+\mu|D v|I)=f^k(v)$ for $|x|<R$. Thus we turn to prove that $u(x)\leq v(x)$ in $B_R(0)$.  

Suppose on the contrary that $u> v$ at some point in $B_R(0)$. 
Then there exists some positive constant $c_0$ such that $u -c_0$
touches $v$ from below at some interior point $x_0$ in $B_R(0)$, i.e.,
$u(x_0)-c_0-v(x_0)=0$ and $u-c_0-v\leq 0$ in
$B_R(0)$. Since $v(x)\to +\infty$ as $|x|\to R$, we can choose $R^{\prime} \in (0, R)$ such that $B_{R^{\prime}}(0)$ contains $x_0$
and $\sup_{\partial B_{R^{\prime}}(0)}(u-c_0-v)<0$.
Let $L[\omega]:=S_k(D^2\omega+\mu|D \omega|I)-f^k(\omega)$. Combining with the monotonicity of $f$ and the definition of subsolution \eqref{subsol}, we have 
\begin{equation*}
	\begin{split}
		L[u -c_0]
		&=S_k(D^2(u -c_0)+\mu|D (u -c_0)|I)-f^k(u -c_0)\\
		&=S_k(D^2u+\mu|D u)|I)-f^k(u -c_0)\\
		&=S_k(D^2u+\mu|D u)|I)-f^k(u)+\left( f^k(u)-f^k(u -c_0)\right) \\
		&\geq 0,\quad x\in   B_{R^{\prime}}(0).
	\end{split}
\end{equation*}
Additionally, we have $L[v]=S_k(D^2v+\mu|D v|I)-f^k(v)=0$ in $B_{R^{\prime}}(0)$.
Therefore by $x_0\in B_{R^{\prime}}(0)$ and the maximum principle, we have 
\begin{equation*}
	0=\sup_{B_{R^{\prime}}(0)}(u-c_0-v)=\sup_{\partial B_{R^{\prime}}(0)}(u-c_0-v)<0,
\end{equation*}
which is impossible.

\textit{Step 2: \eqref{1.1} has a subsolution $u(x)\in \Phi_\mu^k(\mathbb{R}^n)$ if and only if \eqref{Cauchy problem} has a solution $\varphi(r)\in C^2[0,+\infty)$ satisfying \eqref{k-cone}}.
First, we prove the sufficiency.
If for $R=+\infty$, the Cauchy problem \eqref{Cauchy problem} has such a solution $\varphi(r)$, we assume $v(x)=\varphi(r)$. 
By Lemma \ref{Lem2.1} and Lemma \ref{Lem2.2}, we know that for $x \in\mathbb{R}^n$ ,
\begin{equation*}
	S_k(D^2v(x)+\mu|D v(x))|I)=f^k(v(x)),
\end{equation*}
and $\lambda(D^2v(x)+\mu|D v(x))|I) \in \Gamma_{k}$.
Then $v(x)\in \Phi_\mu^k(\mathbb{R}^n)$ is a required solution of equation \eqref{1.1}.

Next, we prove the necessity, that is, if there exists a subsolution $u(x)\in \Phi_\mu^k(\mathbb{R}^n)$ of equation \eqref{1.1}, then we want to prove that  solutions of the Cauchy problem \eqref{Cauchy problem} exist in $[0,+\infty)$ for any initial value $a$.
Suppose that \eqref{Cauchy problem} has no global solution $\varphi(r)$.
It follows from Lemma \ref{Lem2.3} that the Cauchy problem \eqref{Cauchy problem} has a local solution  $\varphi(r)$ existing on some intervals with $\varphi(0)=a$ and $\varphi^{\prime}(0)=0$.  
Then there is a maximal interval $[0, R)$ in which the solution exists. 
By Lemma \ref{Lem2.3}, $\varphi^{\prime}(r)>0$, we know that $\varphi(r)\to +\infty$ as $r\to R$, 
if not, which shows that $\varphi(r)$ is bounded on $[0,R]$, then we can also define an Euler's break line on it to approximate the solution of \eqref{1.1}, which contradicts the maximum interval.
In addition, by Lemma \ref{Lem2.2}, $\varphi(r)$ satisfies \eqref{ODE} and \eqref{k-cone} in $B_R(0)$.
On the other hand, we know from the first step that any subsolution $u(x)\in \Phi_\mu^k(\mathbb{R}^n)$ of equation \eqref{1.1} would satisfy $u(x)\leq\varphi(r)$ for $|x|\leq R$. 
Then we have $u(0)\leq \varphi(0)=a$ particularly. 
Since $a$ is arbitrary, we get a contradiction by choosing
$a = \frac{u(0)}{2}$ if $u(0)>0$, $a = 2u(0)$ if $u(0)<0$ and $a =-1$ if $u(0)=0$.

\textit{Step 3: \eqref{Cauchy problem} has a solution $\varphi(r)\in C^2[0,+\infty)$ satifying \eqref{k-cone} if and only if \eqref{nsc} holds}. 
First, we prove the sufficient condition. 
To simplify the presentation,  we multiply both sides of equation \eqref{ODE} by $\frac{r^k}{(1+\mu r)^{k-1}}$ and unify similar terms,
\begin{equation}\label{mODE}
	\begin{split}
		\frac{r^k}{(1+\mu r)^{k-1}}f^k(\varphi(r))
		&=C_{n-1}^{k-1}r
		\left(\varphi^{\prime\prime}(r)+\mu\varphi^{\prime}(r) \right) 
		\left(\varphi^{\prime}(r) \right)^{k-1} +
		C_{n-1}^k(1+\mu r)
		\left(\varphi^{\prime}(r) \right)^{k}\\
		&=C_{n-1}^{k-1}r\varphi^{\prime\prime}(r)\left(\varphi^{\prime}(r) \right)^{k-1}+
		\left(C_n^k \mu r+ C_{n-1}^k\right) \left(\varphi^{\prime}(r) \right)^{k}.
	\end{split}
\end{equation} 
Suppose that there is no such solution of \eqref{Cauchy problem}. 
However we can find a solution $\varphi(r)$ of \eqref{Cauchy problem} existing in $C^2[0,R)$ with $a=0$ by Lemma \ref{Lem2.3}, and satisfying $\varphi(r)\to +\infty$ as $r\to R$, where $[0,R)$ is the maximal existence interval and $R<+\infty$. 
It follows from Lemma \ref{Lem2.2} that $\varphi(r)$ satisfies the simplified equation \eqref{mODE} for $r\in [0,R)$.

For the case $1\leq k \leq n$, $\mu\ge 0$, since $\varphi^{\prime}(r)>0$, it follows from \eqref{mODE} that 
\begin{equation}\label{3.2}
	C_{n-1}^{k-1}\varphi^{\prime\prime}(r)\left(\varphi^{\prime}(r) \right)^{k-1}<\left( \frac{r}{1+\mu r}\right) ^{k-1}f^k(\varphi(r)), \quad r\in [0,R).
\end{equation}
Multiplying both sides of \eqref{3.2} by $\varphi^{\prime}(r)$, we get 
\begin{equation*}
\begin{split}
	\varphi^{\prime\prime}(r)\left(\varphi^{\prime}(r) \right)^{k}
	&<C \left( \frac{r}{1+\mu r}\right) ^{k-1}f^k(\varphi(r))\varphi^{\prime}(r)\\
	&<C \left( \frac{R}{1+\mu R}\right) ^{k-1}f^k(\varphi(r))\varphi^{\prime}(r),  \quad r\in [0,R),
\end{split}
\end{equation*}
where $C=(C_{n-1}^{k-1})^{-1}$ and it may be taken as different constant depending on $n$ and $k$ in the sequel.
Integrating the above inequality from $0$ to $r$, combining with $\varphi(0)=0$ and $\varphi^{\prime}(0)=0$, we have 
\begin{equation*}
	\left(\varphi^{\prime}(r) \right)^{k+1}<C\left( \frac{R}{1+\mu R}\right) ^{k-1}
	\int_{0}^{\varphi(r)}f^k(t) \,{\rm d}t.
\end{equation*}
By a simple calculation, we get
\begin{equation}\label{3.3le}
	\left( \int_{0}^{\varphi(r)}f^k(t) \,{\rm d}t\right)^{-\frac{1}{k+1}} \,{\rm d}\varphi 
<C\left( \frac{R}{1+\mu R}\right) ^{\frac{k-1}{k+1}}\,{\rm d}r.
\end{equation}
Integrating \eqref{3.3le} from $0$ to $R$ and using the fact that $\varphi(0)=0$ and $\varphi(R)=+\infty$, we get 
\begin{equation}
	\int_{0}^{+\infty}\left( \int_{0}^{\varphi}f^k(t) \,{\rm d}t\right)^{-\frac{1}{k+1}} \,{\rm d}\varphi 
	<C\left( \frac{R}{1+\mu R}\right) ^{\frac{k-1}{k+1}}R
	<+\infty.
\end{equation}
which is a contradiction to \eqref{nsc}.

For the case $k =1$, $\mu< 0$, \eqref{mODE} becomes
\begin{equation}\label{k =1 case 1}
\begin{split}
r f(\varphi(r)) = & r\varphi^{\prime\prime}(r) + [n\mu r + (n-1)]\varphi^\prime(r) \\
> & r\varphi^{\prime\prime}(r) + n\mu r \varphi^\prime(r),
\end{split}
\end{equation}
where $\varphi^\prime(x)>0$ is used in the inequality. Multiplying both sides of \eqref{k =1 case 1} by $e^{2n\mu r}\varphi^\prime(r)/r$, we get
\begin{equation*}
\left[\frac{1}{2}e^{2n\mu r} (\varphi^\prime(r))^2\right]^\prime < e^{2n\mu r}f(\varphi(r))\varphi^\prime(r).
\end{equation*}
Integrating from $0$ to $r$, using $\varphi^\prime(0)=0$ and $\mu<0$, we get
\begin{equation*}
\frac{1}{2}e^{2n\mu r} (\varphi^\prime(r))^2 < \int_{0}^r e^{2n\mu s}f(\varphi(s))\varphi^\prime(s) {\rm d}s \le \int_{0}^{\varphi(r)} f(t){\rm d}t,
\end{equation*}
which leads to
\begin{equation}\label{k = 1 case 2}
\left ( \int_{0}^\varphi f(t) {\rm d}t\right)^{-\frac{1}{2}} {\rm d}\varphi < \sqrt{2} e^{-n\mu R} {\rm d} r.
\end{equation}
Integrating \eqref{k = 1 case 2} from $0$ to $R$ and using the fact that $\varphi(0)=0$ and $\varphi(R)=+\infty$, we get 
\begin{equation}
\int_{0}^{+\infty}\left( \int_{0}^{\varphi}f(t) \,{\rm d}t\right)^{-\frac{1}{2}} \,{\rm d}\varphi < \int_{0}^R \sqrt{2} R e^{-n\mu R}{\rm d}r = \sqrt{2} R e^{-n\mu R} < +\infty,
\end{equation}
which is a contradiction to \eqref{nsc} for $k=1$.

Next, we prove the necessary condition. 
Assume on the contrary that \eqref{nsc} does not hold, namely,
\begin{equation}\label{3.3ass}
	\int^{+\infty}\left(\int_{0}^{\tau}f^k(t)\,{\rm d}t \right) ^{-\frac{1}{k+1}}\,{\rm d}\tau<+\infty.
\end{equation}
We claim that 
\begin{equation}\label{3.3geq}
	\frac{f(t)}{t}\to +\infty, \quad {\rm as} \ t \to +\infty.
\end{equation}
To prove \eqref{3.3geq}, we let 
\begin{equation*}
	g(\tau)=\left(\int_{0}^{\tau}f^k(t)\,{\rm d}t \right) ^{-\frac{1}{k+1}}.
\end{equation*}
By \eqref{3.3ass}, we have $\int^{+\infty}g(\tau)\,{\rm d}\tau<+\infty$, which deduces that $\int_s^{+\infty}g(\tau)\,{\rm d}\tau\to 0$ as $s\to +\infty$.
Moreover, noticing that $g(\tau)$ is non-increasing in $(0, +\infty)$, we have 
\begin{equation*}
	0<\tau g(\tau)\leq 2\int_{\frac{\tau}{2}}^{\tau}g(s)\,{\rm d}s
	<2\int_{\frac{\tau}{2}}^{+\infty}g(s)\,{\rm d}s \to 0, \quad \rm{as} \ \tau\to +\infty.
\end{equation*}
Thus 
\begin{equation*}
	\tau g(\tau)=\tau \left(\int_{0}^{\tau}f^k(t)\,{\rm d}t\right) ^{-\frac{1}{k+1}}\to 0,\quad {\rm as} \ \tau\to 0.
\end{equation*}
Combining this with the monotonically non-decreasing property of $f$, we have
\begin{equation*}
\left(\frac{ f(\tau)}{\tau} \right)^k =\frac{\tau f^k(\tau)}{\tau^{k+1}}
	\geq\frac{\int_{0}^{\tau}f^k(t)\,{\rm d}t}{\tau^{k+1}}\to +\infty, \quad \rm{as} \ \tau\to +\infty.
\end{equation*}
Hence \eqref{3.3geq} holds. Therefore there exists $t_1 > 0$ such that 
\begin{equation}\label{3.5}
	f(t)>t-\varphi(0)
\end{equation}
for $t>t_1$. 
Since  
\begin{equation*}
	\begin{split}
		\varphi^{\prime}(r)
		&=\left(
		\frac{r^{k-n}}{e^{n\mu r } }
		\int_{0}^{r} \frac{k}{C_{n-1}^{k-1}}
		\frac{e^{n\mu s}s^{n-1}}{(1+\mu s)^{k-1}} 
		f^k(\varphi(s)) \,{\rm d}s
		\right)^{\frac{1}{k}}\\
		&\geq \left( \frac{k}{C_{n-1}^{k-1}}\right)^{\frac{1}{k}} f(\varphi(0))
		\left(
		\frac{r^{k-n}}{e^{n\mu r } }
		\int_{0}^{r}
		\frac{e^{n\mu s}s^{n-1}}{(1+\mu s)^{k-1}}  \,{\rm d}s
		\right)^{\frac{1}{k}},
	\end{split}
\end{equation*}
where \eqref{Cauchy problem} and $\varphi(r)>\varphi(0)$ are used in the above inequality, we can fix a constant $r_1>0$ such that for $r>r_1$,
\begin{equation}\label{3.7}
		\varphi^{\prime}(r)
		\geq \left( \frac{k}{C_{n-1}^{k-1}}\right)^{\frac{1}{k}} 
		\left( \frac{r_1}{1+\mu r_1} \right) ^{\frac{k-1}{k}}
		f(\varphi(0))
		\left(
		\frac{r^{k-n}}{e^{n\mu r } }
		\int_{r_1}^{r}
		e^{n\mu s}s^{n-k}  \,{\rm d}s
		\right)^{\frac{1}{k}}.
\end{equation}
We notice that
\begin{equation}\label{3.8}
	\begin{split}
		&\int_{r_1}^{r} e^{n\mu s} s^{n-k}\,{\rm d}s\\
	    =&\left. \frac{1}{n \mu}s^{n-k}e^{n\mu s}\right| _{r_1}^{r}
	    +(-1)\left. \frac{n-k}{(n \mu)^2}s^{n-k-1}e^{n\mu s}\right| _{r_1}^{r}
	    +(-1)^2\left. \frac{(n-k)(n-k-1)}{(n \mu)^3}s^{n-k-2}e^{n\mu s}\right| _{r_1}^{r}
	    +\cdots\\
	    &+\left.(-1)^{n-k-1} \frac{\prod_{i=2}^{n-k}i}{(n \mu)^{n-k}}se^{n\mu s}\right| _{r_1}^{r}
	    +\left.(-1)^{n-k} \frac{\prod_{i=1}^{n-k}i}{(n \mu)^{n-k+1}}e^{n\mu s}\right| _{r_1}^{r}\\
	    =& M+ \frac{1}{n \mu}r^{n-k}e^{n\mu r}
	    +(-1)\frac{n-k}{(n \mu)^2}r^{n-k-1}e^{n\mu r}
	    +(-1)^2\frac{(n-k)(n-k-1)}{(n \mu)^3}r^{n-k-2}e^{n\mu r}+
	    \cdots\\
	    &+(-1)^{n-k-1} \frac{\prod_{i=2}^{n-k}i}{(n \mu)^{n-k}}re^{n\mu r}
	    +(-1)^{n-k} \frac{\prod_{i=1}^{n-k}i}{(n \mu)^{n-k+1}}e^{n\mu r},
	\end{split}
\end{equation}
where the method of integration by parts is used in the second line and the constant $M$ is bounded depending only on $k$, $n$ and $r_1$. 
Substituting \eqref{3.8} into \eqref{3.7}, we get
\begin{equation}\label{3.9}
	\begin {split}
	& \varphi^{\prime}(r)
		\geq 
	     C
	     \left(\frac{1}{e^{n\mu r} r^{n-k}}
	     \int_{r_1}^{r} e^{n\mu s} s^{n-k}\,{\rm d}s\right) ^\frac{1}{k}\\
		=&C\left[ \frac{n\mu M}{e^{n\mu r} r^{n-k}}+
		\left( 1+(-1)\frac{n-k}{n \mu r}+(-1)^2
		\frac{(n-k)(n-k-1)}{(n \mu r)^2}+\cdots
		+(-1)^{n-k} \frac{\prod_{i=1}^{n-k}i}{(n \mu r)^{n-k}}\right)\right] ^\frac{1}{k},
	\end{split}
\end{equation}
where the constant $C$ denotes $\frac{1}{n \mu}\left( \frac{k}{C_{n-1}^{k-1}}\right)^{\frac{1}{k}} 
\left( \frac{r_1}{1+\mu r_1} \right) ^{\frac{k-1}{k}} f(\varphi(0))$.
By choosing $r_2>1$ larger if necessary such that for $r>r_2$,
\begin{equation}\label{3.10}
	0<\frac{n\mu M}{e^{n\mu r} r^{n-k}}+
	\left( 1+(-1)\frac{n-k}{n \mu r}+(-1)^2
	\frac{(n-k)(n-k-1)}{(n \mu r)^2}+\cdots
	+(-1)^{n-k} \frac{\prod_{i=1}^{n-k}i}{(n \mu r)^{n-k}}\right)<1,
\end{equation}
and applying \eqref{3.10} to \eqref{3.9}, we get for $r>r_2$,
\begin{equation*}
		\varphi^{\prime}(r)\geq 
		C\left[ \frac{n\mu M}{e^{n\mu r} r^{n-k}}+
		\left( 1+(-1)\frac{n-k}{n \mu r}+(-1)^2
		\frac{(n-k)(n-k-1)}{(n \mu r)^2}+\cdots
		+(-1)^{n-k} \frac{\prod_{i=1}^{n-k}i}{(n \mu r)^{n-k}}\right)\right].
\end{equation*}
Integrating the above inequality from $r_2$ to $r$, we have for $r>r_2$,
\begin{equation*}
	\begin{split}
		\varphi(r)
		\geq&\varphi(r_2)+ C
		\left[  \int_{r_2}^{r}\frac{n\mu M}{e^{n\mu s} s^{n-k}}\,{\rm d}s+
		(r-r_2)+(-1)\frac{n-k}{n\mu}\left( \ln r-\ln r_2\right)  + (-1)^2\frac{(n-k)(n-k-1)}{(n \mu)^2} \right.\\
		&\left.
		\left( \frac{1}{r_2}-\frac{1}{r}\right)+\cdots +(-1)^{n-k} \frac{\prod_{i=1}^{n-k}i}{(n-k-1)(n \mu)^{n-k}}\left( \frac{1}{r_2^{n-k-1}}-\frac{1}{r^{n-k-1}}\right) \right] .
	\end{split}
\end{equation*}
It follows that 
$\varphi(r)\to +\infty$ as $r\to +\infty$,   
there always exists $r^{\prime }>\max \{r_1,r_2\}$ such that $\varphi(r)>t_1$ for $r>r^{\prime }$. Thus for $r>r^{\prime }$, according to \eqref{3.5} we have 
\begin{equation}\label{relation}
	f(\varphi(r))>\varphi(r)-a.
\end{equation}

\vspace{2mm}

For the case $1\leq k \leq n$, $0\leq \mu< \mu_0$,
since $\varphi^{\prime}(r)>0$ for $r>0$, it follows from \eqref{mODE} that 
\begin{equation*}
	C_{n-1}^{k-1}\varphi^{\prime\prime}(r)\left(\varphi^{\prime}(r) \right)^{k-1}
	<\left( \frac{r}{1+\mu r}\right) ^{k-1}f^k(\varphi(r)).
\end{equation*}
Multiplying the above inequality by $\varphi^{\prime}(r)$ and integrating on $r$, we have 
\begin{equation}\label{3.12}
	\begin{split}
		C_{n-1}^{k-1}\left(\varphi^{\prime}(r) \right)^{k+1}
		&< (k+1) \int_{0}^{r} \left( \frac{s}{1+\mu s}\right) ^{k-1}f^k(\varphi(s)) \varphi^{\prime}(s) \,{\rm d}s\\
		&< (k+1) \left( \frac{r}{1+\mu r}\right) ^{k-1}  f^k(\varphi(r) )
		\left( \varphi(r)-a\right) \\
		&<(k+1) \left( \frac{r}{1+\mu r}\right) ^{k-1}  f^{k+1}(\varphi(r) )
	\end{split}
\end{equation}
for $r>r^{\prime }$, where $\varphi^{\prime}(0)=0$, the monotonicity of $f$ and \eqref{relation} are used. 
By \eqref{3.12}, for $r>r^{\prime }$, we have 
\begin{equation}\label{3.14}
	\begin{split}
		\left(C_n^k \mu r+ C_{n-1}^k\right) \left(\varphi^{\prime}(r) \right)^{k}
		&<\left(C_n^k \mu r+ C_{n-1}^k\right) 
		\left( \frac{k+1}{C_{n-1}^{k-1}} \left( \frac{r}{1+\mu r}\right) ^{k-1} f^{k+1}(\varphi(r) )\right) ^{\frac{k}{k+1}} \\
		&=\left( \frac{k+1}{C_{n-1}^{k-1}}\right) ^{\frac{k}{k+1}} 
		\frac{(C_n^k \mu r+ C_{n-1}^k)(1+\mu r)^\frac{k-1}{k+1}}{r^{\frac{2k}{k+1}}}
		\frac{r^k}{(1+\mu r)^{k-1}}f^k(\varphi(r)).
	\end{split}
\end{equation}
Since 
\begin{equation}\label{3.15}
	\lim_{r\to +\infty}\frac{(C_n^k \mu r+ C_{n-1}^k)(1+\mu r)^\frac{k-1}{k+1}}{r^{\frac{2k}{k+1}}}=C_n^k\mu^{\frac{2k}{k+1}},
\end{equation}
there exists some constant $r^{\prime \prime }>r^{\prime }$ such that
\begin{equation*}
	\frac{(C_n^k \mu r+ C_{n-1}^k)(1+\mu r)^\frac{k-1}{k+1}}{r^{\frac{2k}{k+1}}}
	\leq C_n^k\left( \frac{\mu +\mu_0}{2}\right) ^{\frac{2k}{k+1}}
\end{equation*}
for $r>r^{\prime \prime }$. 
Combining \eqref{3.14} and \eqref{3.15}, for $r>r^{\prime \prime }$, we have 
\begin{equation}\label{3.16}
	\begin{split}
		\left(C_n^k \mu r+ C_{n-1}^k\right) \left(\varphi^{\prime}(r) \right)^{k}
		&<\left( \frac{k+1}{C_{n-1}^{k-1}}\right) ^{\frac{k}{k+1}} 
		 C_n^k\left( \frac{\mu +\mu_0}{2}\right) ^{\frac{2k}{k+1}}
		\frac{r^k}{(1+\mu r)^{k-1}}f^k(\varphi(r))\\
		&=\left( \frac{k+1}{C_{n-1}^{k-1}}\right) ^{\frac{k}{k+1}} 
		C_n^k(1-\delta)\mu_0^{\frac{2k}{k+1}}
		\frac{r^k}{(1+\mu r)^{k-1}}f^k(\varphi(r))\\
		&=(1-\delta)\frac{r^k}{(1+\mu r)^{k-1}}f^k(\varphi(r)),
	\end{split}
\end{equation}
where the constant $\delta:=1-\left(\frac{\mu+\mu_0}{2\mu_0} \right)^\frac{2k}{k+1}\in (0,1)$ and $\mu_0$ in \eqref{mu 0} are used in the last equality.
Thus, by \eqref{mODE} and \eqref{3.16}, for $r>r^{\prime \prime }$, we have
\begin{equation}\label{3.11}
	C_{n-1}^{k-1}\varphi^{\prime\prime}(r)\left(\varphi^{\prime}(r) \right)^{k-1}
	>\delta  \left( \frac{r}{1+\mu r}\right) ^{k-1}f^k(\varphi(r)).
\end{equation}
Multiplying both sides of \eqref{3.11} by $\varphi^{\prime}(r)$ and integrating on $r$,  we get
\begin{equation*}
	\left(\varphi^{\prime}(r) \right)^{k+1}
	>\left(\varphi^{\prime}(r) \right)^{k+1}-\left(\varphi^{\prime}(r^{\prime \prime)} \right)^{k+1}>
	C  \int_{r^{\prime \prime }}^{r} \left( \frac{s}{1+\mu s}\right) ^{k-1}f^k(\varphi(s)) \varphi^{\prime}(s) \,{\rm d}s,	
\end{equation*}
where we use the fact that $\varphi^{\prime}(r)>0$ for $r>0$ in the above inequality. 
Fixing a constant $r_3>r^{\prime \prime }$ such that for $r>r_3$,
\begin{equation*}
	\begin{split}
	\left(\varphi^{\prime}(r) \right)^{k+1}
	&>C  \int_{r_3}^{r} \left( \frac{s}{1+\mu s}\right) ^{k-1}f^k(\varphi(s)) \varphi^{\prime}(s) \,{\rm d}s\\
	&>C    \left( \frac{r_3}{1+\mu r_3}\right) ^{k-1}
	 \int_{\varphi(r_3)}^{\varphi(r)} f^k(t) \,{\rm d}t.
	\end{split}
\end{equation*}
By a simple calculation,  we have 
\begin{equation*}
	\left( \int_{\varphi(r_3)}^{\varphi(r)} f^k(t) \,{\rm d}t\right) ^{-\frac{1}{k+1}} \,{\rm d}\varphi > C\,{\rm d}r,
\end{equation*}
where the constant $C$ also depends on $r_3$.
Integrating again and noticing that $f$ is monotonically non-decreasing, we have for $\tau >2 \varphi(r_3)$,
\begin{equation}\label{3.3res}
	\begin{aligned}
    C(r-r_3)&<\int_{\varphi(r_3)}^{\varphi(r)}\left( \int_{\varphi(r_3)}^{\tau} f^k(t) \,{\rm d}t\right) ^{-\frac{1}{k+1}} \,{\rm d} \tau \\
    &\leq\int_{\varphi(r_3)}^{+\infty}\left( \int_{\varphi(r_3)}^{\tau} f^k(t) \,{\rm d}t\right) ^{-\frac{1}{k+1}} \,{\rm d} \tau \\
    &\leq\int_{\varphi(r_3)}^{+\infty}\left( \int_{0}^{\frac{\tau}{2}} f^k(t) \,{\rm d}t\right) ^{-\frac{1}{k+1}} \,{\rm d} \tau \\
    &=2\int_{\frac{\varphi(r_3)}{2}}^{+\infty}\left( \int_{0}^{\tau} f^k(t) \,{\rm d}t\right) ^{-\frac{1}{k+1}} \,{\rm d} \tau .
	\end{aligned}
\end{equation}
Since $r$ can be arbitrarily large, \eqref{3.3res} contradicts with \eqref{3.3ass}.

\hspace{2mm}

For the case $k=1$, $\mu <0$, equation \eqref{mODE} reduces to 
\begin{equation}\label{k=1}
	r f(\varphi(r))= r \varphi^{\prime \prime }(r)+(n\mu r+n-1) \varphi^{\prime}(r).
\end{equation}
Since $f>0$, it follows from \eqref{Cauchy problem} that 
\begin{equation*}
	\varphi^{\prime}(r)=\frac{1}{e^{n\mu r}r^{n-1}}\int_{0}^re^{n\mu s}s^{n-1}f(\varphi(r))\,{\rm d}s>0
\end{equation*} 
for $r>0$. 
If $r>\frac{1-n}{n\mu}$, we have $n\mu r+n-1<0$. 
Hence, we have 
\begin{equation}\label{3.20}
	(n\mu r+n-1) \varphi^{\prime}(r)<0
\end{equation}
for $r>\frac{1-n}{n\mu}$. Applying \eqref{3.20} to \eqref{k=1}, we have $\varphi^{\prime \prime }(r)>f(\varphi(r))$ for $r>\frac{1-n}{n\mu}$ .
Multiplying both sides by $\varphi^{\prime}(r)$ and integrating from $\frac{1-n}{n\mu}$ to $r$, we get
\begin{equation*}
	\begin{split}
	\left( \varphi^{\prime}(r)\right) ^2
	>&\left( \varphi^{\prime}(r)\right) ^2-\left( \varphi^{\prime}(\frac{1-n}{n\mu})\right) ^2\\
	>&2\int_{\frac{1-n}{n\mu}}^{r} f(\varphi(s))  \varphi^{\prime}(s) \,{\rm d}s\\
	=&2 \int_{\varphi(\frac{1-n}{n\mu})}^{\varphi(r)} f(t) \,{\rm d}t.
	\end{split}
\end{equation*}
By a simple calculation,  we have 
\begin{equation*}
	\left( \int_{\varphi(\frac{1-n}{n\mu})}^{\varphi(r)} f(t) \,{\rm d}t\right) ^{-\frac{1}{2}} \,{\rm d}\varphi > \sqrt{2}\,{\rm d}r.
\end{equation*}
Then similar to \eqref{3.3res} we get a contradiction,  which completes the proof.

\textit{Step 4: Conclusion}. Combining Steps 2 and 3, we deduce that equation \eqref{1.1} has a subsolution $u(x)\in \Phi^k(\mathbb{R}^n)$
if and only if the Cauchy problem \eqref{Cauchy problem} has a solution $\varphi(r)\in C^2[0,+\infty)$ satisfying \eqref{k-cone} for any initial value $a$, which is equivalent to the condition \eqref{nsc}.
\end{proof}
\begin{remark}\label{Remark 3.1}
	It is worth noting that for the proof of the sufficient condition in the third step of Theorem \ref{Th1.2}, we have no upper bound restriction on the parameter $\mu$. This effectively means that for either the case $k=1$, $\mu\in(-\infty,+\infty)$; or the case $2\leq k\leq n$, $\mu \in [0,+\infty)$, if the condition \eqref{nsc} holds, then there exists a solution to the Cauchy problem \eqref{Cauchy problem}. 
\end{remark}

Next, we will prove Theorem \ref{Th1.1} with the help of the proof of Theorem \ref{Th1.2}. 

\begin{proof}[Proof of Theorem \ref{Th1.1}]
    We first prove the existence result. For either the case $k=1$, $\mu\in(-\infty,+\infty)$; or the case $2\leq k\leq n$, $\mu \in [0,+\infty)$, if the condition \eqref{nsc} holds, by using the sufficient condition in the third step in the proof of Theorem \ref{Th1.2}, we obtain the existence of solutions to \eqref{Cauchy problem}, see Remark \ref{Remark 3.1}. Then with the aid of the second step in the proof of Theorem \ref{Th1.2}, we actually obtain the existence of subsolutions to \eqref{1.1}. 
	Alternatively, for $2\leq k\leq n$, $\mu \in [0,+\infty)$, the existence of admissible subsolutions to \eqref{1.1} can be directly obtained by the existence result of the standard $k$-Hessian equation in \cite{JB}, that is if equation $S_k^{\frac{1}{k}}(D^2u)=f(u)$ admits an entire admissible subsolution, then the existence of entire admissible subsolutions to equation \eqref{1.1} is also obtained. The key technique is the use of the inequality 
	$$S_k^{\frac{1}{k}}(D^2u+\mu|D u|I)\geq S_k^{\frac{1}{k}}(D^2u) $$ 
	for $\mu \geq 0$. 
	
	Next, we prove the nonexistence result.
   For either the case $k=1$, $\mu\in(-\infty,\mu_0)$; or the case $2\leq k\leq n$, $\mu \in [0,\mu_0)$, it follows from Theorem \ref{Th1.2} that equation \eqref{1.1} admits an entire admissible subsolution if and only if \eqref{nsc} holds. 
   Therefore, if $f$ satisfies
   \begin{equation*}
   	\int^{+\infty}\left(\int_{0}^{\tau}f^k(t)\,{\rm d}t \right) ^{-\frac{1}{k+1}}\,{\rm d}\tau<+\infty,
   \end{equation*} 
   then \eqref{nsc} fails, which shows that equation \eqref{1.1} admits no subsolution $u\in \Phi_{\mu}^k(\mathbb{R}^n) $.
   For the case $2\leq k\leq n$, $\mu\in(-\infty,0)$, it follows from Remark \ref{re2.1} that equation \eqref{1.1} admits no entire subsolution satisfying \eqref{k-cone}.
\end{proof}

To end this section, according to the proof of Theorem \ref{Th1.2}, we give the proof of Theorem \ref{Th1.3}. Note that $f(t)$ is degenerate at $t=0$, we turn to seek positive solutions to avoid the degeneracy.

\begin{proof}[Proof of Theorem \ref{Th1.3}]
For either $k=1$, $\mu\in(-\infty,\mu_0)$, or $2\leq k\leq n$, $\mu \in [0,\mu_0)$, if $f$ satisfies \eqref{f0}, by \eqref{Cauchy problem}, we have  
\begin{equation*}
    	\varphi^{\prime}(r)=\left(
    	\frac{r^{k-n}}{e^{n\mu r } }
    	\int_{0}^{r} \frac{k}{C_{n-1}^{k-1}}
    	\frac{e^{n\mu s}s^{n-1}}{(1+\mu s)^{k-1}} 
    	f^k(\varphi(s)) \,{\rm d}s
    	\right)^{\frac{1}{k}}\geq 0
\end{equation*}
for $r>0$, which shows that $\varphi$ is monotonically non-decreasing. 
In order to discuss the radial solutions of \eqref{1.1} in a convex cone, for any positive number $a$, we let the Cauchy problem \eqref{Cauchy problem} have initial value $\varphi(0)=a>0$.
Combining with the monotonically non-decreasing property of $\varphi$, we get
\begin{equation}\label{3.21}
   	\varphi(r)\geq \varphi(0)=a>0 ,\quad r>0.
\end{equation}
Since $f$ satisfies \eqref{f0}, we obtain  
   \begin{equation*}
   	f(\varphi(r))\geq f(a)>0,\quad r>0.
\end{equation*}
Then we have $\varphi^{\prime}(r)>0$ for $r>0$.
Similar to Lemma \ref{Lem2.2}, for $1\leq p\leq k$, we have $S_p(D^2u+\mu|D u|I)>0$, which shows that \eqref{k-cone} holds. Hence for this degenerate case, we restrict our discussion within the context of positive solutions. 
We only need to replace the initial value of the Cauchy problem \eqref{Cauchy problem}  with any positive number $a$ and the rest of the proof is similar to the proof of Theorem \ref{Th1.2}.
\end{proof}
\vspace{3mm}

\section{Applications}\label{Section 4}
In this section, we give two applications to Theorem \ref{Th1.2} and Theorem \ref{Th1.3} respectively, which provide the necessary and sufficient conditions on the existence and nonexistence of entire subsolutions to specific equation \eqref{1.1} with two kinds of nonlinearities. 
We also give two examples to verify the existence results.  

The applications of the global solvability are given in the following corollaries.
\begin{Corollary}\label{Corollary4.1}
Let $\alpha\geq 0$ and $f(t)=e^{\alpha t}$, then for either the case $k=1$, $\mu\in(-\infty,\mu_0)$; or the case $2\leq k\leq n$, $\mu \in [0,\mu_0)$, equation \eqref{1.1} admits a subsolution $u\in \Phi_{\mu}^k(\mathbb{R}^n)$ if and only if $\alpha= 0$. For $1\leq k\leq n$, $\mu \in (-\infty,\mu_0)$, if $\alpha>0$, then there exists no entire admissible subsolution of \eqref{1.1}.
\end{Corollary}
\begin{proof}[Proof of Corollary \ref{Corollary4.1}]
	For either the case $k=1$, $\mu\in(-\infty,\mu_0)$; or the case $2\leq k\leq n$, $\mu \in [0,\mu_0)$, if $\alpha=0$, then $f(t)\equiv1$, which satisfies the condition in Theorem \ref{Th1.2}. Moreover, we have  
	\begin{equation*}
		\int^{+\infty}\left(\int_{0}^{\tau}f^k(t)\,{\rm d}t \right) ^{-\frac{1}{k+1}}\,{\rm d}\tau=\int^{+\infty}\left(\int_{0}^{\tau}\,{\rm d}t \right) ^{-\frac{1}{k+1}}\,{\rm d}\tau=+\infty.
	\end{equation*}
Hence, by Theorem \ref{Th1.2}, there exists a subsolution of \eqref{1.1} if and only if $\alpha=0$.

For either the case $k=1$, $\mu\in(-\infty,\mu_0)$; or the case $2\leq k\leq n$, $\mu\in[0,\mu_0)$, if $\alpha>0$, we have $f^{\prime}(t)=\alpha e^{\alpha t}>0$, which also satisfies the condition in Theorem \ref{Th1.2}. Moreover, we have 
    \begin{equation*}
    	\begin{split}
    		\int^{+\infty}\left(\int_{0}^{\tau}f^k(t)\,{\rm d}t \right) ^{-\frac{1}{k+1}}\,{\rm d}\tau
    		=&\int^{+\infty}\left(\int_{0}^{\tau}e^{k\alpha t}\,{\rm d}t \right) ^{-\frac{1}{k+1}}\,{\rm d}\tau\\
    		=&\sqrt[k+1]{k\alpha}\int^{+\infty}\frac{1}{\left( e^{k\alpha\tau }-1\right) ^{\frac{1}{k+1}}}\,{\rm d}\tau.
    	\end{split}
    \end{equation*}
Since 
\begin{equation*}
	\lim_{\tau\to +\infty}\frac{\tau^2}{(e^{k\alpha\tau }-1)^{\frac{1}{k+1}}}
	=\lim_{\tau\to +\infty}\left( \frac{\tau^{2(k+1)}}{e^{k\alpha\tau}-1}\right) ^{\frac{1}{k+1}}=\left(\lim_{\tau\to +\infty} \frac{\tau^{2(k+1)}}{e^{k\alpha\tau}-1}\right) ^{\frac{1}{k+1}}=0,
\end{equation*}
by Cauchy’s criterion for the infinite integral with nonnegative integrand, then the integral 
$$\int^{+\infty}\frac{1}{\left( e^{k\alpha\tau }-1\right) ^{\frac{1}{k+1}}}\,{\rm d}\tau \ {\rm converges}, $$
which implies that 
\begin{equation*}
	\int^{+\infty}\left(\int_{0}^{\tau}f^k(t)\,{\rm d}t \right) ^{-\frac{1}{k+1}}\,{\rm d}\tau<+\infty,
\end{equation*}
for $f(t)=e^{\alpha t}$ with $\alpha >0$.
By Theorem \ref{Th1.2}, we know that equation \eqref{1.1} admits no entire admissible subsolution.
For $2\leq k\leq n$, $\mu <0$, if $\alpha>0$, it follows from Remark \ref{re2.1} that there is no entire subsolution satisfying \eqref{k-cone} to equation \eqref{1.1}. To sum up, we get the desired result.
\end{proof}

\begin{Corollary}\label{Corollary4.2}
	Let $q\geq 0$ and 
	\begin{equation*}
		f(t)=
		\left\{
		\begin{split}
			&t^q,&t>0,\\
			&0,&t\leq0,
		\end{split}
		\right.
	\end{equation*}
then for either the case $k=1$, $\mu\in(-\infty,\mu_0)$; or the case $2\leq k\leq n$, $\mu \in [0,\mu_0)$, equation \eqref{1.1} admits a positive subsolution $u\in \Phi_{\mu}^k(\mathbb{R}^n)$ if and only if $q\leq1$.
\end{Corollary}

\begin{proof}[Proof of Corollary \ref{Corollary4.2}]
	For either the case $k=1$, $\mu\in(-\infty,\mu_0)$; or the case $2\leq k\leq n$, $\mu \in [0,\mu_0)$, if $0\leq q \leq 1$, then we have 
	\begin{equation*}
		\begin{split}
			\int^{+\infty}\left(\int_{0}^{\tau}f^k(t)\,{\rm d}t \right) ^{-\frac{1}{k+1}}\,{\rm d}\tau
			=&\int^{+\infty}\left(\int_{0}^{\tau}t^{kq}\,{\rm d}t \right) ^{-\frac{1}{k+1}}\,{\rm d}\tau\\
			=&\sqrt[k+1]{kq+1}\int^{+\infty} \tau^{-\frac{kq+1}{k+1}}\,{\rm d}\tau.
		\end{split}
	\end{equation*}
Since 
\begin{equation*}
\lim_{\tau\to +\infty}\tau^{\frac{kq+1}{k+1}} \tau^{-\frac{kq+1}{k+1}}=1
\end{equation*}
and $\frac{kq+1}{k+1}\leq 1$, by Cauchy’s criterion for the infinite integral with nonnegative integrand, then the infinite integral $\int^{+\infty} \tau^{-\frac{kq+1}{k+1}}\,{\rm d}\tau$ diverges, therefore we have 
\begin{equation*}
\int^{+\infty}\left(\int_{0}^{\tau}f^k(t)\,{\rm d}t \right) ^{-\frac{1}{k+1}}\,{\rm d}\tau=+\infty.
\end{equation*}
Besides, we have $\frac{kq+1}{k+1}>1$ for $q>1$. Again by Cauchy’s criterion, the infinite integral $\int^{+\infty} \tau^{-\frac{kq+1}{k+1}}\,{\rm d}\tau$ converges, that is,
\begin{equation*}
\int^{+\infty}\left(\int_{0}^{\tau}f^k(t)\,{\rm d}t \right) ^{-\frac{1}{k+1}}\,{\rm d}\tau<+\infty, 
\end{equation*}
which does not satisfy the condition \eqref{nsc}.
By Theorem \ref{Th1.3}, we know that equation \eqref{1.1} admits a positive subsolution $u\in \Phi_{\mu}^k(\mathbb{R}^n)$ if and only if $0\leq q\leq1$.
\end{proof}

Moreover, we give an example to the existence result in Theorem \ref{Th1.1} with a power-type nonlinearity.
	\begin{Example}\label{Example 4.1}
	In $\mathbb{R}^n$, for $k=1, \cdots, n$, $\mu\ge 0$, then the function $u(x):=e^{Ax^2}>0$ with $A\geq \frac{1}{2}\left(\frac{1}{C_n^k} \right)^{\frac{1}{k}}$ is an entire subsolution to equation
	\begin{equation}\label{4.1}
		S_k^{\frac{1}{k}}(D^2u+\mu|D u|I)= u^\alpha
	\end{equation}
	with  $\alpha\leq1$.
	\end{Example}

	\begin{proof}
	Letting $\omega(x)=e^{Ax^2}$,
	then 
	\begin{equation*}
		\mu |D \omega|I=\mu 2Ae^{Ax^2}|x|\delta_{ij},
	\end{equation*}
	and
	\begin{equation*}
		D^2\omega+\mu |D \omega|I=4A^2e^{Ax^2}\left[ x_ix_j+\frac{1}{2A}\left( 	1+\mu |x|\right) \delta_{ij}\right].
	\end{equation*}
	It follows from \eqref{ab} that  $\lambda(D^2\omega+\mu |D \omega|I)$ can be denoted as 
	\begin{equation*}
		\lambda(D^2\omega+\mu |D \omega|I)=
		\left( 4A^2e^{Ax^2}\left( |x|^2+\frac{1}{2A}(1+\mu |x|)\right) , 	2Ae^{Ax^2}(1+\mu |x|),\cdots, 2Ae^{Ax^2}(1+\mu |x|)\right).
	\end{equation*}
	By the definition of $k$-Hessian operator, we have
	\begin{equation*}
		\begin{split}
			&S_k\left( \lambda(D^2\omega+\mu |D \omega|I)\right)\\
			=&C_{n-1}^{k-1} 4A^2e^{Ax^2}\left( |x|^2+\frac{1}{2A}(1+\mu |x|)\right) 
			\left[ 2Ae^{Ax^2}(1+\mu |x|)\right]^{k-1}+C_{n-1}^k \left[ 2Ae^{Ax^2}(1+\mu |x|)\right]^k\\
			=&C_{n-1}^{k-1} 4A^2e^{Ax^2} |x|^2\left[ 2Ae^{Ax^2}(1+\mu |x|)\right]^{k-1}+
			C_n^k\left[ 2Ae^{Ax^2}(1+\mu |x|)\right]^k\\
			\geq& C_n^k\left[ 2Ae^{Ax^2}\right]^k 
			\geq e^{kAx^2},
		\end{split}
	\end{equation*}
	provided $A\geq \frac{1}{2}\left(\frac{1}{C_n^k} \right)^{\frac{1}{k}} $. 
	Therefore we have 
	$S_k\left( \lambda(D^2\omega+\mu |D \omega|I)\right)\geq \omega^k$
	all $k=1,\cdots,n$. 
	For $\alpha\leq 1$, combining with $\omega=e^{Ax^2}>1$, we have 
	\begin{equation}\label{4.similar}
		S_k\left( \lambda(D^2\omega+\mu |D \omega|I)\right)\geq \omega^k \geq 	\omega^{k\alpha},
	\end{equation}
	which implies that $\omega$ is a positive entire subsolution to equation \eqref{4.1} with $\alpha \le 1$.
	\end{proof}
Note that the index $\alpha$ of power function here can be negative, whose range is larger than the existence condition in Theorem \ref{Th1.1}. This means that when the nonlinearity is monotonically decreasing, the existence result takes effect as well. 

In the end, we give an example showing that even though $\mu<0$ for the case $k=1$, we can still explicitly construct an entire subsolution to \eqref{1.1}.

	\begin{Example}\label{Example 4.2}
	In $\mathbb{R}^n$, for $k=1$, $\mu<0$, then the function $u(x):=e^{Ax^2}>0$ with $A\geq \frac{1}{2n}+\frac{1}{8}n\mu^2$ is an entire subsolution to equation 
	\begin{equation}\label{4.2}
		S_1(D^2u+\mu |D u|I)= u^\alpha,
	\end{equation}
	with  $\alpha\leq1$.
	\end{Example}
	
	\begin{proof}
	Letting $\omega(x):=e^{Ax^2}$, 
	then $|D \omega|=2Ae^{Ax^2}|x|$ and  $\Delta\omega=4A^2e^{Ax^2}x^2+2Ane^{Ax^2}$.
	By the definition of $k$-Hessian operator, we have 
	\begin{equation}\label{4.3}
		\begin{split}
			S_1\left( \lambda(D^2\omega+\mu |D \omega|I)\right)
			&=tr(D^2\omega+\mu |D \omega|I)\\
			&= \Delta\omega+n\mu |D \omega|\\
			&=e^{Ax^2}\left(  4A^2x^2+2An+2An\mu|x|\right),
		\end{split}
	\end{equation}
	where $tr(D^2\omega+\mu |D \omega|I)$ denotes the trace of the matrix $D^2\omega+\mu |D \omega|I$.
	By Cauchy's inequality, we have 
	\begin{equation}\label{4.000}
		\begin{split}
			-2An\mu|x|=2(A|x|)(-n\mu) \leq 4A^2x^2+\frac{1}{4}n^2\mu ^2 \le 4A^2x^2+ (2An-1),
		\end{split}
	\end{equation}
	provided $A>0$ and $\mu<0$, where $A\geq \frac{1}{2n}+\frac{1}{8}n\mu^2$ is used in the last inequality. From \eqref{4.000}, we get
        $$4A^2x^2+2An+2An\mu|x| \geq 1,$$
        which leads to
        	\begin{equation}\label{4.4}
	e^{Ax^2}\left(  4A^2x^2+2An+2An\mu|x|\right) \geq e^{Ax^2}.
	\end{equation} 
	Therefore, we have proved that 
	$S_1\left( \lambda(D^2\omega+\mu |D \omega|I)\right)\geq \omega$.
	Similar to \eqref{4.similar}, we have 
        $$S_1\left( \lambda(D^2\omega+\mu |D \omega|I)\right)\geq \omega \ge \omega^{\alpha},$$
	for $\alpha \le 1$.
	\end{proof}
\vspace{3mm}

In fact, it is Example \ref{Example 4.2} that motivates us to consider the ``$k=1, \mu<0$'' case separately, throughout the paper.

\bibliographystyle{amsplain}

\begin{thebibliography}{s2}
\bibitem{SMA} S. Alarc\'{o}n, J. Garc\'{i}a-Meli\'{a}n and A. Quaas, \textit{Keller-Osserman type conditions for some elliptic problems with gradient terms}, J. Differential Equations, 252 (2012), 886--914.	

\bibitem{BG} C. Bandle, E. Giarrusso, \textit{Boundary blow up for semilinear elliptic equations with nonlinear gradient terms}, Adv. Differential Equations, 1 (1996), 133--150.

\bibitem{BJL} J. Bao, X. Ji and H. Li, \textit{Existence and nonexistence theorem for entire subsolutions of $k$-Yamabe type equations}, J. Differential Equations, 253 (2012), 2140--2160.

\bibitem{CNS} L. Caffarelli, L. Nirenberg and J. Spruck, \textit{The Dirichlet problem for nonlinear second order elliptic equations III: Functions of the eigenvalues of the Hessian}, Acta Math., 155 (1985), 261--301.

\bibitem{CDLV1} I. Capuzzo Dolcetta, F. Leoni and A. Vitolo, \textit{Entire subsolutions of fully nonlinear degenerate elliptic equations}, Bull. Inst. Math. Acad. Sin., 9 (2014), 147--161.

\bibitem{CDLV2} I. Capuzzo Dolcetta, F. Leoni and A. Vitolo, \textit{On the inequality $F(x, D^2u)\geq f(u)+g(u)|Du|^q$}, Math. Ann., 365 (2016), 423--448.

\bibitem{CDLV3} I. Capuzzo Dolcetta, F. Leoni and A. Vitolo, \textit{Generalized Keller-Osserman conditions for fully nonlinear degenerate elliptic equations}, J. Math. Sci., 260 (2022), 469--479.

\bibitem{CW} K. Chou, X.-J. Wang, \textit{A variational theory of the Hessian equation}, Comm. Pure Appl. Math., 54 (2001), 1029--1064.

\bibitem{C} J. Cui, \textit{Existence and nonexistence of entire $k$-convex radial solutions to Hessian type system}, Adv. Difference Equ., (2021), Paper No. 462, 9pp.

\bibitem{D1} L. Dai, H. Li, \textit{Entire subsolutions of Monge-Am\`{e}re type equations}, Commun. Pure Appl. Anal. 19 (2020), 19--30.

\bibitem{D2} L. Dai, \textit{Existence and nonexistence of subsolutions for augmented Hessian equations}, Discrete Contin. Dyn. Syst., 40 (2020), 579--596.

\bibitem{G2} B. Guan, \textit{The Dirichlet problem for Hessian equations on Riemannian manifolds}, Calc. Var. Partial Differential Equations, 8 (1999), 45--69.

\bibitem{G} B. Guan, \textit{Second-order estimates and regularity for fully nonlinear elliptic equations on Riemannian manifolds}, Duke Math. J., 163 (2014), 1491--1524.

\bibitem{GH} C.E. Gutiérrez, Q. Huang, \textit{The refractor problem in reshaping light beams}, Arch. Rat. Mech. Anal., 193 (2009), 423--443.

\bibitem{HJL} Y. Huang, F. Jiang and J. Liu, \textit{Boundary $C^{2,\alpha}$ estimates for Monge-Amp\`{e}re type equations}, Adv. Math., 281 (2015), 706--733.

\bibitem{JJD} J. Ji, F. Jiang and B. Dong, \textit{On the solutions to weakly coupled system of $k_i$-Hessian equations}, J. Math. Anal. Appl., 513 (2022), Paper No. 126217, 25pp.

\bibitem{JB} X. Ji, J. Bao, \textit{Necessary and sufficient conditions on solvability for Hessian inequalities}, Proc. Am. Math. Soc., 138 (2010), 175--188.

\bibitem{J} H. Jian, \textit{Hessian equations with infinite Dirichlet boundary value}, Indiana Univ. Math. J., 55 (2006), 1045--1062.

\bibitem{JTY} F. Jiang, N.S. Trudinger and X. Yang, \textit{On the Dirichlet problem for Monge-Amp\`{e}re type equations}, Calc. Var. Partial Differential Equations,  49 (2014), 1223--1236.

\bibitem{JT1} F. Jiang, N.S. Trudinger, \textit{Oblique boundary value problems for augmented Hessian equations I,II,III}, Bull. Math. Sci., 8 (2018), 353--411; Nonlinear Anal., 154 (2017), 148--173; Comm. Partial Differential Equations, 44 (2019), 708--748.

\bibitem{JT2} F. Jiang, N.S. Trudinger, \textit{On the Dirichlet problem for general augmented Hessian equations}, J. Differential Equations, 269 (2020), 5204--5227.
	
	
\bibitem{K}	J.B. Keller, \textit{On solutions of $\Delta u = f(u)$}, Comm. Pure Appl. Math., 10 (1957), 503--510.

\bibitem{LLions} J.M. Lasry, P.L. Lions, \textit{Nonlinear elliptic equations with singular boundary conditions and stochastic control with state constraints. I. The model problem}, Math. Ann., 283 (1989), 583--630.
	
\bibitem{LL} A. Li, Y. Li, \textit{On some conformally invariant fully nonlinear equations. II. Liouville, Harnack and Yamabe}, Acta Math., 195 (2005), 117--154.

\bibitem{LTW} J. Liu, N.S. Trudinger and X.-J. Wang, \textit{Interior $C^{2,\alpha}$ regularity for potential functions in optimal transportation}, Comm. Partial Differential Equations, 35 (2010), 165--184.

\bibitem{MTW} X. Ma, N.S. Trudinger and X.-J. Wang, \textit{Regularity of potential functions of the optimal transportation problem}, Arch. Rat. Mech. Anal., 177 (2005), 151--183.

\bibitem{O}	R. Osserman, \textit{On the inequality $\Delta u \geq f(u)$}, Pacific J. Math., 7 (1957), 1641--1647.

\bibitem{PV} A. Porretta, L. Veron, \textit{Symmetry of large solutions of nonlinear elliptic equations in a ball}, J. Funct. Anal., 236 (2006), 581--591.

\bibitem{S} P. Salani, \textit{Boundary blow-up problems for Hessian equations}, Manuscripta Math., 96 (1998), 281--294.

\bibitem{Sch} F. Schulz, \textit{Regularity theory for quasilinear elliptic systems and Monge-Amp\`ere equations in two dimensions}, Lecture Notes in Mathematics, 1445. Springer-Verlag, Berlin, 1990.

\bibitem{T} N. Trudinger, \textit{On the Dirichlet problem for Hessian equations}, Acta Math., 175 (1995), 151--164.


\bibitem{TW} N. Trudinger, X.-J. Wang, \textit{Hessian measures I,II,III}, Topol. Methods Nonlinear Anal., 10 (1997), 225-239; Ann. of Math., 150 (1999), 579--604; J. Funct. Anal., 193 (2002), 1--23.

\bibitem{U} J. Urbas, \textit{Hessian equations on compact Riemannian manifolds}, Nonlinear problems in mathematical physics and related topics, II, Int. Math. Ser., 2 (2002), 367--377.

\bibitem{V} J. Viaclovsky, \textit{Conformal geometry, contact geometry, and the calculus of variations}, Duke Math. J., 101 (2000), 283--316.

\bibitem{W1} X.-J. Wang, \textit{On the design of a reflector antenna II}, Calc. Var. Partial Differential Equations, 20 (2004), 329--341.

\bibitem{ZQ} Z. Zhang, Z. Qi, \textit{On a power-type coupled system of Monge-Amp\`{e}re equations}, Topol. Methods Nonlinear Anal., 46 (2015), 717--729.

\bibitem{ZZ} Z. Zhang, S. Zhou, \textit{Existence of entire positive $k$-convex radial solutions to Hessian equations and systems with weights}, Appl. Math. Lett., 50 (2015), 48--55.

\end{thebibliography}

\end{document}